\documentclass[12pt]{amsart}
\usepackage{}
\usepackage{amsmath}
\usepackage{amsfonts}
\usepackage{amssymb}
\usepackage[all,cmtip]{xy}           
\usepackage{xspace}

\usepackage{bbding}
\usepackage{txfonts}
\usepackage[shortlabels]{enumitem}
\usepackage{ifpdf}
\ifpdf
\usepackage[colorlinks,final,backref=page,hyperindex]{hyperref}
\else
\usepackage[colorlinks,final,backref=page,hyperindex,hypertex]{hyperref}
\fi
\usepackage{tikz}
\usepackage[active]{srcltx}

\makeatletter

\newtheorem{theorem}{Theorem}[section]
\newtheorem{prop}[theorem]{Proposition}
\newtheorem{lemma}[theorem]{Lemma}
\newtheorem{coro}[theorem]{Corollary}
\newtheorem{prop-def}{Proposition-Definition}[section]
\newtheorem{conjecture}[theorem]{Conjecture}
\newtheorem{question}[theorem]{Question}
\theoremstyle{definition}
\newtheorem{defn}[theorem]{Definition}

\newtheorem{remark}[theorem]{Remark}

\newcommand{\nc}{\newcommand}

\nc{\delete}[1]{{}}
\nc{\mmargin}[1]{}

\nc{\mlabel}[1]{\label{#1}}  
\nc{\mcite}[1]{\cite{#1}}  
\nc{\mref}[1]{\ref{#1}}  
\nc{\mbibitem}[1]{\bibitem{#1}} 

\delete{
	\nc{\mlabel}[1]{\label{#1}  
		{\hfill \hspace{1cm}{\bf{{\ }\hfill(#1)}}}}
	\nc{\mcite}[1]{\cite{#1}{{\bf{{\ }(#1)}}}}  
	\nc{\mref}[1]{\ref{#1}{{\bf{{\ }(#1)}}}}  
	\nc{\mbibitem}[1]{\bibitem [\bf #1]{#1}} 
}

 \font\cyrs=wncyr7

\newcommand{\bk}{{\mathbf{k}}}

\nc{\ideal}{\mathrm{I}}

\nc{\free}[1]{\cald_\lambda\langle{#1}\rangle}
\nc{\freec}[1]{\cald_\lambda[{#1}]}
\nc{\freed}[1]{\cald_\lambda({#1})}
\nc{\freecd}[1]{\calc\cald_\lambda({#1})}
\nc{\freel}[1]{\call_\lambda({#1})}
\nc{\vep}{\varepsilon}
\nc{\bin}[2]{ (_{\stackrel{\scs{#1}}{\scs{#2}}})}  
\nc{\binc}[2]{(\!\! \begin{array}{c} \scs{#1}\\
		\scs{#2} \end{array}\!\!)}  
\nc{\bincc}[2]{  ( {\scs{#1} \atop
		\vspace{-1cm}\scs{#2}} )}  
\nc{\oline}[1]{\overline{#1}}
\nc{\mapm}[1]{\lfloor\!|{#1}|\!\rfloor}
\nc{\bs}{\bar{S}}
\nc{\la}{\longrightarrow}
\nc{\ot}{\otimes}
\nc{\rar}{\rightarrow}
\nc{\dar}{\downarrow}
\nc{\dap}[1]{\downarrow \rlap{$\scriptstyle{#1}$}}
\nc{\defeq}{\stackrel{\rm def}{=}}
\nc{\dis}[1]{\displaystyle{#1}}
\nc{\dotcup}{\ \displaystyle{\bigcup^\bullet}\ }
\nc{\hcm}{\ \hat{,}\ }
\nc{\hts}{\hat{\otimes}}
\nc{\hcirc}{\hat{\circ}}
\nc{\lleft}{[}
\nc{\lright}{]}
\nc{\curlyl}{\left \{ \begin{array}{c} {} \\ {} \end{array}
	\right .  \!\!\!\!\!\!\!}
\nc{\curlyr}{ \!\!\!\!\!\!\!
	\left . \begin{array}{c} {} \\ {} \end{array}
	\right \} }
\nc{\longmid}{\left | \begin{array}{c} {} \\ {} \end{array}
	\right . \!\!\!\!\!\!\!}
\nc{\ora}[1]{\stackrel{#1}{\rar}}
\nc{\ola}[1]{\stackrel{#1}{\la}}
\nc{\scs}[1]{\scriptstyle{#1}} \nc{\mrm}[1]{{\rm #1}}
\nc{\dirlim}{\displaystyle{\lim_{\longrightarrow}}\,}
\nc{\invlim}{\displaystyle{\lim_{\longleftarrow}}\,}
\nc{\dislim}[1]{\displaystyle{\lim_{#1}}} \nc{\colim}{\mrm{colim}}
\nc{\mvp}{\vspace{0.3cm}} \nc{\tk}{^{(k)}} \nc{\tp}{^\prime}
\nc{\ttp}{^{\prime\prime}} \nc{\svp}{\vspace{2cm}}
\nc{\vp}{\vspace{8cm}}
\nc{\modg}[1]{\!<\!\!{#1}\!\!>}
\nc{\intg}[1]{F_C(#1)}
\nc{\lmodg}{\!<\!\!}
\nc{\rmodg}{\!\!>\!}
\nc{\cpi}{\widehat{\Pi}}
\nc{\ssha}{{\mbox{\cyrs X}}} 
\nc{\tsha}{{\mbox{\cyrt X}}}
\nc{\shpr}{\diamond}    
\nc{\labs}{\mid\!}
\nc{\rabs}{\!\mid}

\nc{\on}{on \xspace}
\nc{\tforall}{\text{\ \  for all }}
\nc{\gsb}{Gr\"obner-Shirshov basis\xspace}
\nc{\gsbs}{Gr\"obner-Shirshov bases\xspace}
\nc{\ann}{\mrm{ann}}
\nc{\Aut}{\mrm{Aut}}
\nc{\br}{\mrm{bre}}
\nc{\can}{\mrm{can}}
\nc{\Cont}{\mrm{Cont}}
\nc{\rchar}{\mrm{char}}
\nc{\cok}{\mrm{coker}}
\nc{\de}{\mrm{dep}}
\nc{\dtf}{{R-{\rm tf}}}
\nc{\dtor}{{R-{\rm tor}}}

\nc{\Div}{{\mrm Div}}
\nc{\End}{\mrm{End}}
\nc{\Ext}{\mrm{Ext}}
\nc{\Fil}{\mrm{Fil}}
\nc{\Fr}{\mrm{Fr}}
\nc{\Frob}{\mrm{Frob}}
\nc{\Gal}{\mrm{Gal}}
\nc{\GL}{\mrm{GL}}
\nc{\Hom}{\mrm{Hom}}
\nc{\hsr}{\mrm{H}}
\nc{\hpol}{\mrm{HP}}
\nc{\id}{\mrm{id}}
\nc{\im}{\mrm{im}}
\nc{\Id}{\mrm{Id}}
\nc{\DI}{\mrm{DI}}
\nc{\DIrr}{\mrm{DIrr}}
\nc{\Irr}{\mrm{Irr}}
\nc{\incl}{\mrm{incl}}
\nc{\lex}{\mrm{lex}}
\nc{\NLSW}{\mrm{NLSW}}
\nc{\Lie}{\mrm{Lie}}
\nc{\mchar}{\rm char}
\nc{\mpart}{\mrm{part}}
\nc{\ql}{{\QQ_\ell}}
\nc{\qp}{{\QQ_p}}
\nc{\rank}{\mrm{rank}}
\nc{\rcot}{\mrm{cot}}
\nc{\rdef}{\mrm{def}}
\nc{\rdiv}{{\rm div}}
\nc{\rtf}{{\rm tf}}
\nc{\rtor}{{\rm tor}}
\nc{\res}{\mrm{res}}
\nc{\SL}{\mrm{SL}}
\nc{\Spec}{\mrm{Spec}}
\nc{\supp}{\mrm{supp}}
\nc{\tor}{\mrm{tor}}
\nc{\Tr}{\mrm{Tr}}
\nc{\tr}{\mrm{tr}}

\nc{\bfk}{{\bf k}}
\nc{\bfone}{{\bf 1}}
\nc{\bfzero}{{\bf 0}}
\nc{\detail}{\marginpar{\bf More detail}
	\noindent{\bf Need more detail!}
	\svp}
\nc{\Diff}{\mathbf{Diff}}
\nc{\gap}{\marginpar{\bf Incomplete}\noindent{\bf Incomplete!!}
	\svp}
\nc{\FMod}{\mathbf{FMod}}
\nc{\Int}{\mathbf{Int}}
\nc{\Mon}{\mathbf{Mon}}
\nc{\remarks}{\noindent{\bf Remarks: }}
\nc{\Rep}{\mathbf{Rep}}
\nc{\Rings}{\mathbf{Rings}}
\nc{\Sets}{\mathbf{Sets}}

\nc{\BA}{{\mathbb A}}   \nc{\CC}{{\mathbb C}}
\nc{\DD}{{\mathbb D}}   \nc{\EE}{{\mathbb E}}
\nc{\FF}{{\mathbb F}}   \nc{\GG}{{\mathbb G}}
\nc{\HH}{{\mathbb H}}   \nc{\LL}{{\mathbb L}}
\nc{\NN}{{\mathbb N}}   \nc{\PP}{{\mathbb P}}
\nc{\QQ}{{\mathbb Q}}   \nc{\RR}{{\mathbb R}}
\nc{\TT}{{\mathbb T}}   \nc{\VV}{{\mathbb V}}
\nc{\ZZ}{{\mathbb Z}}   \nc{\TP}{\widetilde{P}}

\nc{\cala}{{\mathcal A}}    \nc{\calc}{{\mathcal C}}
\nc{\cald}{\mathcal{D}}     \nc{\cale}{{\mathcal E}}
\nc{\calf}{{\mathcal F}}    \nc{\calg}{{\mathcal G}}
\nc{\calh}{{\mathcal H}}    \nc{\cali}{{\mathcal I}}
\nc{\call}{{\mathcal L}}    \nc{\calm}{{\mathcal M}}
\nc{\caln}{{\mathcal N}}    \nc{\calo}{{\mathcal O}}
\nc{\calp}{{\mathcal P}}    \nc{\calr}{{\mathcal R}}
\nc{\cals}{{\mathcal S}}    \nc{\calt}{{\Omega}}
\nc{\calv}{{\mathcal V}}    \nc{\calw}{{\mathcal W}}
\nc{\calx}{{\mathcal X}}

\nc{\fraka}{{\mathfrak a}}
\nc{\frakb}{{\mathfrak b}}
\nc{\frakc}{{\mathfrak c}}
\nc{\frakg}{{\frak g}}
\nc{\frakB}{{\frak B}}
\nc{\frakm}{{\frak m}}
\nc{\frakM}{{\frak M}}
\nc{\frakp}{{\frak p}}
\nc{\frakW}{{\frak W}}
\nc{\frakX}{{\frak X}}
\nc{\frakS}{{\frak S}}
\nc{\frakA}{{\frak A}}
\nc{\fraks}{{\frak s}}
\nc{\fraku}{{\frak u}}
\nc{\frakv}{{\frak v}}
\nc{\frakx}{{\frak x}}
\nc{\frakI}{{\frak I}}

\nc{\lir}[1]{\textcolor{red}{\underline{Li:}#1 }}

\begin{document}

\title[Free differential algebras and Gr\"obner-Shirshov bases]{Construction of free differential algebras by extending Gr\"{o}bner-Shirshov bases}

\author[Li Guo]{Li Guo}
\address{Department of Mathematics and Computer Science, Rutgers University, Newark, NJ 07102, USA}
\email{liguo@rutgers.edu}

\author[Yunnan Li]{Yunnan Li}
\address{School of Mathematics and Information Science, Guangzhou University, Waihuan Road West 230, Guangzhou 510006, China}
\email{ynli@gzhu.edu.cn}

\date{\today}

\begin{abstract}
As a fundamental notion, the free differential algebra on a set is concretely constructed as the polynomial algebra on the differential variables. Such a construction is not known for the more general notion of the free differential algebra on an algebra, from the left adjoint functor of the forgetful functor from differential algebras to algebras, instead of sets. In this paper we show that generator-relation properties of a base algebra can be extended to the free differential algebra on this base algebra. More precisely, a Gr\"obner-Shirshov basis property of the base algebra can be extended to the free differential algebra on this base algebra, allowing a Poincar\'e-Birkhoff-Witt type basis for these more general free differential algebras. Examples are given as illustrations.
\end{abstract}

\subjclass[2010]{13Nxx, 12H05, 16S15, 16S10, 16W99}

\keywords{differential algebra, free differential algebra, Gr\"obner basis, Gr\"{o}bner-Shirshov basis}

\maketitle

\tableofcontents

\allowdisplaybreaks

\section{Introduction}

A differential algebra is an associative algebra equipped with a linear operator satisfying the Leibniz rule modeled after the differential operator in analysis. Indeed, the study of differential algebra originated from the algebraic study of differential equations pioneered by Ritt in the 1930s~\cite{Rit,Rit1}. In the following decades, the theory has been fully developed to comprise differential Galois theory, differential algebraic geometry and differential algebraic groups~\cite{Kol,Mag,PS}, with broad connections to areas in mathematics and mathematical physics~\cite{FLS,MS}.

Traditionally, differential algebra works in the context of commutative algebras and fields. There the free objects are generated by sets, realized as differential polynomial algebras, which are simply polynomials generated by differential variables from the generating sets. In recent years, this traditional framework of differential algebra has been extended in several directions which motivated the present study.

On the one hand, the commutativity condition is removed to include naturally arisen algebras such as path algebras and to have a more meaningful differential Lie algebra theory generalizing the classical relationship between associative algebra and Lie algebra~\cite{GL,Poi,Poi1}.
On the other hand, the Leibniz rule is generalized to include the differential quotient $\frac{f(x+\lambda)-f(x)}{\lambda}$ before taking the limit, leading to the notion of a differential algebra of weight $\lambda$~\cite{GK}, closely related to the extensively studied difference algebra~\cite{Coh,Lev,PS1}.

Furthermore, as an integral analog of differential algebras, Rota-Baxter algebras are algebraic abstraction of the integral and summation operators, originated from the probability study of G. Baxter~\cite{Ba} in 1960 and experienced rapid expansion with applications in combinatorics, number theory and mathematical physics. See~\cite{Guo,Guo1,GK1} and references therein. Free commutative Rota-Baxter algebras on sets were constructed by Rota and Cartier many years ago. A more general construction of free Rota-Baxter algebras was obtained as the left adjoint functor of the forgetful functor from the category of commutative Rota-Baxter algebras to that of commutative algebras~\cite{GK1} instead of sets. This more general construction was obtained from a generalization of the shuffle product, called the mixable shuffle product, closely related to the quasi-shuffle product which plays a pivot role in the study of multiple zeta values. In the noncommutative context, the (noncommutative) Rota-Baxter algebras on an algebra were constructed by bracketed words and planar rooted trees~\cite{EG,Guo}.

Thus it is natural to study the free differential algebras generated by an algebra, in both the commutative and noncommutative contexts and for an arbitrary weight, as the left adjoint functor of the forgetful functor from the category of commutative (resp. noncommutative) differential algebras to that of commutative (resp. noncommutative) algebras. Surprisingly, in contrast to the extensive study of free Rota-Baxter algebras generated by algebras as noted above, free differential algebras generated by algebras were not studied until recently by~\cite{Poi,Poi1} where such free differential algebras were expressed as quotients.

The goal of this paper is to gain a better understanding of the free differential algebra on an algebra that is comparable to the free differential algebra on a set, by constructing a canonical linear basis of the free differential algebra, rather than just to express the free differential algebra as a quotient. Naturally in order for this to work, conditions on the generating algebra should be expected. More precisely, starting with an algebra with a canonical basis, we would like to obtain a basis for the free differential algebra on this algebra. So this is an analogy of the Poincar\'e-Birkhoff-Witt Theorem that gives a basis of the universal enveloping algebra of a Lie algebra with a given basis. Our approach is to apply the method of Gr\"obner-Shirshov bases.
Beginning with an algebra with a presentation by generator and relations for which the relation ideal is equipped by a Gr\"obner basis or noncommutative Gr\"obner basis, we derive a similar presentation of the free differential algebra that this algebra generates. The latter will give rise of a desired Poincar\'e-Birkhoff-Witt type bases. Since we will consider algebras with differential operators in both the commutative and noncommutative contexts, we will work with the more general notion of Gr\"obner-Shirshov bases.
For earlier approaches on differential Gr\"obner bases and characteristic sets for differential commutative algebras of weight zero, see~\mcite{Ai,CF,Ma,Ol} and references therein.

The method of Gr\"{o}bner-Shirshov bases provides a powerful tool to deal with ideals of free objects in various algebraic structures.
It provides an algorithmic way to derive a linear basis for an algebra presented by generators and relations.
The readers are referred to the survey \cite{BC} for a general review of the method of Gr\"{o}bner-Shirshov bases and their relationship with Gr\"obner bases.
So far there have been several studies on differential algebras and their related generalizations via the theory of Gr\"{o}bner-Shirshov bases \cite{BCQ,CCL,GGR,GSZ,QC}. In particular, free differential algebras and free differential Rota-Baxter algebras over sets~\cite{GK} can also be obtained via Gr\"{o}bner-Shirshov bases~\cite{BCQ}. In a closely related structure, free integro-differential algebras on sets in both the commutative and noncommutative contexts are also obtained by Gr\"obner-Shirshov bases~\mcite{GG,GRR,GGR}.

The outline of this paper is as follows. In Section~\mref{sec:cd} we first recall the notion of a differential algebra (with weight) and give our preliminary construction of the free differential algebra on an algebra. We then present the composition-diamond lemmas on Gr\"{o}bner-Shirshov bases for associative algebras and differential algebras, including their commutative counterparts. In Section~\mref{sec:bases}, we apply the composition-diamond lemma for differential algebras to obtain Gr\"obner-Shirshov bases for free differential algebras on algebras in both the noncommutative case (Theorem~\mref{mtgs}) and the commutative case (Theorem~\ref{mtgsc}). These results on \gsbs allow us to obtain canonical bases for these free objects (Propositions~\mref{prop:fdab} and \mref{prop:fcdab}). In a nutshell, we show that a \gsb of an algebra can be ``differentially" extended to a differential \gsb of the free differential algebra on this algebra, in all the cases except for the ``classical" one, namely for differential commutative algebras with weight zero, when there are obstructions of such extension.
As demonstrations, examples are given in Section~\mref{ss:exam} for free differential algebras on special algebras, including commutative algebras with one generator and some finite group algebras.

\smallskip

\noindent
{\bf Notation.} Throughout the paper, we fix a base field $\bk$ of characteristic 0. All modules, algebras, homomorphisms and tensor products are taken over $\bfk$ unless otherwise specified.

\section{Differential algebras and composition-diamond lemmas}
\mlabel{sec:cd}

In this section, we provide preliminary results, motivation and the composition-diamond lemmas that will be used in our study of \gsbs construction of free differential algebras on algebras later in this paper.

\subsection{Differential algebras and free differential algebras}
\mlabel{ss:free}
We first recall the notion of a differential algebra with weight~\mcite{GK}.
\begin{defn}
Let $\lambda\in \bk$ be fixed. A {\bf differential $\bk$-algebra of weight $\lambda$} (also called a {\bf $\lambda$-differential $\bk$-algebra}) is an associative $\bk$-algebra $R$ together with a linear operator $d:R\rar R$ such that
\begin{equation}
d(xy)=d(x)y+x d(y)+\lambda d(x)d(y) \tforall\, x,y\in R.
\mlabel{diff}
\end{equation}
If $R$ is unital, it further requires that
\begin{equation}\label{difu}
d(1_R)=0.
\end{equation}
Such an operator is called a {\bf differential operator of weight $\lambda$} or a {\bf derivation of weight $\lambda$}. It is also called a {\bf $\lambda$-differential operator} or a {\bf $\lambda$-derivation} in short.
A {\bf homomorphism of differential algebras} is an algebra homomorphism commutating with the derivations. The category of differential algebras (resp. differential commutative algebras) of weight $\lambda$ is denoted by DA$_\lambda$ (resp. CDA$_\lambda$).
\end{defn}

Let $(R,d)$ be a differential $\bfk$-algebra of weight $\lambda$. The higher order Leibniz rule of derivatives generalizes to~\cite[Proposition 2.1.(b)]{GK},
\begin{equation}
d^{n}(xy) =
\sum_{j=0}^{n}\sum_{k=0}^{n-j}{n\choose j}{n-j \choose k}
\lambda^{j}d^{n-k}(x)d^{j+k}(y)
=\hspace{-.7cm} \sum_{n=j+k+\ell, j, k, \ell\geq 0} {n\choose j, k, \ell} \lambda^j d^{n-k}(x)d^{j+k}(y)\,
\mlabel{eq:der2}
\end{equation}
for all $x,y \in R, n \in \NN$. Applying this equation repeatedly, we obtain its multi-factor generalization which we record for later use.
\begin{lemma}
Let $R$ be a $\lambda$-differential $\bk$-algebra. Let $x_1,\dots x_{r+1}\in R,\,r\geq1$ and $n\geq0$. Then
\begin{equation}
\begin{split}
d^n(x_1\cdots x_{r+1})
&= \sum_{\substack{n=j_1+k_1+\ell_1\\
j_1+k_1=j_2+k_2+\ell_2\\
 \cdots \\
 j_{r-1}+k_{r-1}=j_r+j_r+\ell_r}} {n\choose j_1,k_1,\ell_1} {j_1+k_1\choose j_2,k_2,\ell_2} \cdots {j_{r-1}+k_{r-1}\choose j_r,k_r,\ell_r} \\
& \qquad \times \lambda^{j_1+\cdots +j_r} d^{n-k_1}(x_1)d^{j_1+k_1-k_2}(x_2)\cdots d^{j_{r-1}+k_{r-1}-k_r}(x_r)d^{j_r+k_r}(x_{r+1}).
\end{split}
\label{dm}
\end{equation}
\end{lemma}

A subset $I$ of a differential algebra $(R,d)$ is called a {\bf differential ideal} of $R$ if $I$ is an ideal of $R$ (as an algebra) such that $d(I)\subseteq I$.
In this case, the quotient algebra $R/I$ is a differential algebra with the induced derivation $\bar{d}$ of $d$ modulo $I$.
For any subset $U$ of $R$, there exists the smallest differential ideal of $R$ containing $U$, called the {\bf differential ideal of $R$ generated by $U$} and denoted by $\DI(U)$. In fact,
by Eq.~(\ref{dm}),
\begin{equation}\label{dix}
\DI(U)=\ideal\left(d^n(u)\,\big|\,u\in U,n\geq0\right).
\end{equation}
Here $\ideal(S)$ or simply $(S)$ denotes the ideal generated by $S$.

\begin{defn}
Let $X$ be a set. The {\bf free differential algebra \on $X$}, denoted by $(\free{X},d_X)$, is a differential algebra $(\free{X},d_X)$ together with a map $i_X:X\rar\free{X}$ satisfying the following universal property. For any differential algebra  $(R,d_R)$ and map $f:X\rar R$, there exists a unique differential algebra homomorphism $\bar{f}:\free{X}\rar R$ such that $f=\bar{f}\circ i_X$.

Also, the {\bf free commutative differential algebra \on $X$} is a differential algebra $(\freec{X},d_X)$ together with a map $i_X:X\rar\freec{X}$ satisfying the same universal property, except that the differential algebra $(R,d_R)$ is commutative.
\end{defn}

When $\lambda=0$, the free commutative differential algebra $\freec{X}$ is traditionally denoted by $\bfk\{X\}$. We use the new notation for easy comparison with the related structures, such as free (noncommutative) differential algebras.

The following fundamental construction of free (commutative) differential algebra \on a set is well-known~\cite{GK,Kol}.
\begin{theorem}
Let $X$ be a set and denote
$$\Delta(X):=X\times \NN= \{ x^{(n)}\, \big|\, x\in X, n\geq 0\}.$$
\begin{enumerate}
\item
Let $\bk\langle\Delta(X)\rangle$ be the noncommutative polynomial algebra on the set $\Delta(X)$.
Define $d_X: \bk\langle\Delta(X)\rangle\to\bk\langle\Delta(X)\rangle$ as follows. Let $w=u_1\cdots u_k$ with $u_i\in\Delta(X)$, $1\leq i\leq k$, be a noncommutative word from the alphabet set $\Delta(X)$.  If $k=1$, so that $w=x^{(n)}\in \Delta(X)$, define $d_X(w):=x^{(n+1)}$. If $k>1$, recursively define
\begin{equation}\label{dprod}
d_X(w):=d_X(u_1)u_2\cdots u_k+u_1d_X(u_2\cdots u_k)+\lambda d_X(u_1)d_X(u_2\cdots u_k).
\end{equation}
Further define $d_X(1):=0$ and extend $d_X$ to $\bk\langle\Delta(X)\rangle$ by linearity.
Then $(\bk\langle\Delta(X)\rangle,d_X)$ is the free (noncommutative) differential algebra $(\free{X},d_X)$ of weight $\lambda$ \on $X$.
\mlabel{it:fda1}
\medskip
\item
Let $\bk[\Delta(X)]$ be the polynomial algebra on the set $\Delta(X)$.
Define $d_X:\bk[\Delta(X)] \to \bk[\Delta(X)]$ on the commutative words from the alphabet set $\Delta(X)$ in the same way as in Item~(\mref{it:fda1}).
Then $(\bk[\Delta(X)],d_X)$ is the free commutative differential algebra $(\freec{X},d_X)$ of weight $\lambda$ \on $X$.
\end{enumerate}
\mlabel{fda}
\end{theorem}

Instead of taking the generating object to be a set, one can also consider free differential algebras \on algebras, as the left adjoint functor of the forgetful functor from the category of differential algebras to that of algebras. This leads to the following notion~\cite{Poi,Poi1}.

\begin{defn} 	
The {\bf free differential algebra of weight $\lambda$ on an algebra $A$}, denoted by $(\freed{A},d_A)$, is a differential algebra $\freed{A}$ with a $\lambda$-derivation $d_A$ and an algebra homomorphism $i_A:A\rar\freed{A}$ satisfying the following universal property. For any differential algebra $(R,d)$ of weight $\lambda$ and an algebra homomorphism $\varphi:A\rar R$, there exists a unique differential algebra homomorphism $\bar{\varphi}:\freed{A}\rar R$ such that $\varphi=\bar{\varphi}\circ i_A$.

When $A$ is commutative, the {\bf free commutative differential algebra of weight $\lambda$} on $A$, denoted by $(\freecd{A},d_A)$, is a commutative differential algebra of weight $\lambda$ with the same universal property, but in the category of commutative differential algebras of weight $\lambda$.
\end{defn}

We note that even if $A$ is a commutative algebra,
$\freed{A}$ may not be commutative and thus is different from $\freecd{A}$. A simple example is when $A$ is the one-variable polynomial algebra $\bk[x]$, for which $\freed{A}=\bk\langle x^{(n)}\,|\,n\geq0\rangle$ and
$\freecd{A}=\bk[x^{(n)}\,|\,n\geq0]$.

\begin{remark}
In the special case when the algebra $A$ is the free $\bk$-algebra $\bk\langle X\rangle$ \on a set $X$, the free differential algebra $\freed{\bk\langle X\rangle}$ on this algebra is precisely the free differential algebra $\free{X}\in \text{DA}_\lambda$ on the set $X$ constructed in Theorem~\ref{fda}, and $\bk\langle X\rangle $ can be embedded into $\free{X}$ as its subalgebra by mapping $x$ to $x^{(0)}$ for any $x\in X$.
\mlabel{rk:algset}
\end{remark}

In general there is no construction of free differential algebras on an algebra $A$ that is as explicit. Instead, as presented in~\cite{Poi1,Poi2}, starting from an algebra $A$, one first takes the free differential algebra $\free{A}$ on the underlying set of $A$. By the universal property of $\free{A}$, the set map $A\to \freed{A}$ induces a differential algebra homomorphism
$$\free{A} \longrightarrow \freed{A}.$$
The homomorphism is surjective and hence gives a presentation of $\freed{A}$ as a quotient $\free{A}/I$ where $I$ is the kernel of the above differential algebra homomorphism. Since the algebra structure on $A$ is ignored in $\free{A}$, both this algebra and the differential ideal $I$ are quite large. Thus the presentation $\freed{A} \cong \free{A}/I$ provides limited information for the free differential algebra $\freed{A}$.

In order to study $\freed{A}$ more effectively, it is desirable to give a presentation of it by smaller quotients, that is, by quotients with smaller numerators and denominators. Furthermore the denominator (the relation ideal) should have good properties, yielding effective computations of the quotient, for example in displaying an explicit basis of $\freed{A}$.

Giving such a presentation is the main purpose of this paper, which we achieve in two steps. First, if $A$ already has a presentation by generators and relations, we use the relation idea of $A$ to obtain a relation ideal of $\freed{A}$. This is the next result. Then after preparational results on \gsbs later in this section, we provide in Section~\mref{sec:bases} more precise information when $A$ is defined by a relation ideal with a Gr\"obner basis (or Gr\"obner-Shirshov basis). See further discussions after the next result.

\begin{prop}
Let $A$ be a $\bk$-algebra and let $A=\bk\langle X\rangle/I_A$ be a presentation of $A$ with generating set $X$ and ideal $I_A$. Let $\frakI_A=\DI(I_A)$ be the differential ideal of $\free{X}$ generated by $I_A$. Then $\freed{A}=\free{X}/\frakI_A$ is a presentation of the free differential algebra \on $A$.

Similarly, let $A=\bk[X]/I_A$ be a commutative algebra and let $\frakI_A=\DI(I_A)$ be the differential ideal of $\freec{X}$ generated by $I_A$. Then the free differential algebra $\freecd{A}$ on $A$ has a presentation $\freecd{A}=\freec{X}/\frakI_A$.
\mlabel{pp:atda}
\end{prop}

\begin{proof}
We only prove the case when $A=\bk\langle X\rangle/I_A$, since the commutative case for $A=\bk[X]/I_A$ is similar to check.

First we fix some notations. Let $j_X:X\rar\bk\langle X\rangle$, $i_X:X\rar\free{X}$ and $\hat{i}_X:\bk\langle X\rangle\rar\free{X}$ be the natural embeddings such that $i_X=\hat{i}_X\circ j_X$. For $A=\bk\langle X\rangle/I_A$, let $\pi_{I_A}:\bk\langle X\rangle\rar A$ and $\pi_{\frakI_A}:\free{X}\rar\free{X}/\frakI_A$ be the
canonical projections.

We organize these maps and the ones that will be introduced later in the proof in the following diagram.
\[\xymatrix@=2em{X\ar@{->}[r]^-{j_X}\ar@{->}[rdd]^-{i_X}&\bk\langle X\rangle\ar@{->}[r]^-{\pi_{I_A}}\ar@{->}[dd]^-{\hat{i}_X}& A\ar@{->}[r]^-{\varphi}\ar@{.>}[rdd]^>>>>>>>{i_A}&(R,d_R)\\
&&&\\&(\free{X},d_X)\ar@{->}[rr]^-{\pi_{\frakI_A}}\ar@{.>}[uurr]_-{\tilde{\varphi}}	 &&(\free{X}/\frakI_A,\bar{d}_X)\ar@{.>}[uu]_-{\bar{\varphi}}}\]

By Eq.~\eqref{dix}, as a differential ideal, \[\frakI_A=\DI\left(\hat{i}_X(I_A)\right)=\left((d_X^n\circ\hat{i}_X)(a)\,\bigg|\,a\in I_A,n\geq0\right).\]
Thus we have $(\pi_{\frakI_A}\circ\hat{i}_X)(I_A)=\{0\}$. Hence, there exists a unique algebra homomorphism $i_A:A\rar\free{X}/\frakI_A$ such that
$i_A\circ\pi_{I_A}=\pi_{\frakI_A}\circ\hat{i}_X$.

Now let $(R,d_R)$ be a differential algebra and $\varphi:A\rar R$ an algebra homomorphism. By the universal property of
$\free{X}$ as the free differential algebra on the algebra $\bk\langle X\rangle$ (Remark~\mref{rk:algset}), there exists a unique homomorphism $\tilde{\varphi}:\free{X}\rar R$ of differential algebras such that $\varphi\circ\pi_{I_A}=\tilde{\varphi}\circ\hat{i}_X$.
Then $\tilde{\varphi}(\hat{i}_X(I_A))=\varphi(\pi_{I_A}(I_A))=\{0\}$. Thus $\tilde{\varphi}(\frakI_A)=\{0\}$. This implies the existence of a unique differential algebra homomorphism $\bar{\varphi}:\free{X}/\frakI_A\rar R$ such that $\tilde{\varphi}=\bar{\varphi}\circ\pi_{\frakI_A}$. Then we have
\[\bar{\varphi}\circ i_A\circ\pi_{I_A}=\bar{\varphi}\circ\pi_{\frakI_A}\circ\hat{i}_X=\tilde{\varphi}\circ\hat{i}_X=\varphi\circ\pi_{I_A}.\]
Since $\pi_{I_A}$ is surjective, it follows that  $\varphi=\bar{\varphi}\circ i_A$. Next suppose that there exists another differential algebra homomorphism $\bar{\varphi}'$ such that $\varphi=\bar{\varphi}'\circ i_A$, then
\[\bar{\varphi}'\circ\pi_{\frakI_A}\circ\hat{i}_X=\bar{\varphi}'\circ i_A\circ\pi_{I_A}=\varphi\circ\pi_{I_A}=\bar{\varphi}\circ i_A\circ\pi_{I_A}=\bar{\varphi}\circ\pi_{\frakI_A}\circ\hat{i}_X.\]
Again by the universal property of $\free{X}$ on $\bk\langle X\rangle$, we have
$\bar{\varphi}'\circ\pi_{\frakI_A}=\bar{\varphi}\circ\pi_{\frakI_A}$. Thus $\bar{\varphi}'=\bar{\varphi}$ since $\pi_{\frakI_A}$ is surjective.

Hence, the free differential algebra $\freed{A}$ on $A$ is isomorphic to $\free{X}/\frakI_A$.
\end{proof}

We will next address the following natural questions from Proposition~\mref{pp:atda}:
If an algebra or commutative algebra $A$ has a certain canonical basis, how to use this basis to derive a canonical basis of the free differential algebra $\freed{A}$ or differential commutative algebra $\freecd{A}$ that $A$ generates?

To make this question more precise and more practical to solve, we use the method of Gr\"obner bases and more generally Gr\"obner-Shirshov bases, beginning with rephrasing the above question in two parts:

\begin{question}
\begin{enumerate}
\item
If a (resp. commutative) algebra $A$ has a presentation $A=\bk\langle X\rangle/I_A$ (resp. $A=\bk[X]/I_A$) such that the ideal $I_A$ has a Gr\"obner basis or Gr\"{o}bner-Shirshov basis $S$, how to extend $S$ to a Gr\"{o}bner-Shirshov basis for the free differential (resp. commutative) algebra $\freed{A}$ (resp. $\freecd{A}$) on $A$?
\mlabel{it:basis1}
\item
Further, if the Gr\"obner-Shirshov basis $S$ gives a canonical basis of $A$, how to use this basis to derive a canonical basis of $\freed{A}$ (resp. $\freecd{A}$)?
\mlabel{it:basis2}
\end{enumerate}
\mlabel{qu:basis}
\end{question}

We will answer Question~\mref{qu:basis} in Section~\mref{sec:bases}, with the technical tools on composition-diamond lemmas for Gr\"obner-Shirshov bases provided in the next subsection.

\subsection{Gr\"{o}bner-Shirshov bases for differential algebras}
\mlabel{ss:gsda}

We will treat the cases of noncommutative differential algebra and commutative differential algebras separately.

For the noncommutative case, the Composition-Diamond (CD) Lemma for free nonunitary differential algebras on sets are given in \cite[Theorem 5.1]{BCQ}.
Further, the composition-diamond (CD) lemma for free nonunitary differential algebras with extra operations are obtained in~\cite{QC}. Here we apply it to the case when there is no extra operations on the differential algebra.
We will recall the notions and results for  applications in later sections.

On the other hand, the CD lemma for free commutative differential algebras has not been established and will be treated in more detail.

\subsubsection{Gr\"{o}bner-Shirshov bases for free differential algebras}

For background, we first review the classical notion of Gr\"{o}bner (-Shirshov) bases for associative algebras~\cite{Ber,Bok,BCM}. The readers can also find the details in \cite[\S 2.1]{BC}.

Denote by $M(X)$ (resp. $S(X)$) the free monoid (resp. semigroup) generated by a set $X$. Any well-order $<$ on $X$ induces a monomial order $<$ on $M(X)$, e.g. the deg-lex order, such that
\[1<u\text{ and }u<v\,\Rightarrow\,wuz<wvz\text{ for any }u,v,w,z\in S(X).\]
Let $\oline{f}\in M(X)$ be the leading term
of any $f\in\bk\langle X\rangle$ with respect to this well-order $<$. A {\bf Gr\"{o}bner-Shirshov basis} in $\bk\langle
X\rangle$ is a subset of monic polynomials whose intersection and inclusion compositions are trivial modulo it.

\begin{lemma} \cite{Ber,BC}
	{\em (Composition-Diamond lemma for associative algebras)} \ Let $S$ be a monic subset of $\bk\langle
	X\rangle$, $\ideal(S)$ be the ideal of $\bk\langle
	X\rangle$ generated by $S$. Let $<$ be a monomial order on $M(X)$. Then the following statements are equivalent:	
	\begin{enumerate}
		\item the set $S$ is a Gr\"{o}bner-Shirshov basis in $\bk\langle
		X\rangle$;		
		\item if $f\in \ideal(S)\backslash\{0\}$, then $\oline{f}=u{\oline{s}}v$ for some $u,v\in M(X)$ and $s\in S$;
		\item the set of {\bf $S$-irreducible words} $\Irr(S): = M(X)\backslash\{u\oline{s}v\,|\,u,v\in M(X),\,s\in S \}$
		is a $\bk$-basis of $\bk\langle X\,|\,S\rangle=\bk\langle X\rangle/\ideal(S)$. Therefore, $\bk\langle X\rangle=\bk\Irr(S)\oplus\ideal(S)$.
	\end{enumerate}
	\mlabel{cdla}
\end{lemma}

We now extend a given well-order $<$ on $X$ to a well-order $\prec$ on $\Delta(X)$ by
\begin{equation}
 x^{(m)}\prec y^{(n)} \text{ if } m<n, \text{ or } m=n \text{ and } x<y \tforall x, y\in X, m, n\geq0.
 \mlabel{eq:diffmon}
\end{equation}
The latter order induces the {\bf deg-lex order} on $M(\Delta(X))$ still denoted by $\prec$ with the {\bf degree} $|u|$ of $u\in M(\Delta(X))$ defined to be the number of letters in $u$ from the alphabet set $\Delta(X)$. Then for $u_1,\dots, u_p,v_1,\dots,v_q\in\Delta(X)$, we define
\[u_1\cdots u_p\prec v_1\cdots v_q\,\Longleftrightarrow\, p>q\text{ or }p=q,\,u_i=v_i,\,i=1,\dots,k-1,\text{ but }u_k\prec v_k\text{ for some }k\leq p.\]

\begin{defn}\mlabel{def:uword}
Let $\star\notin \Delta(X)$ be an extra symbol, $\Delta(X)^\star:=\Delta(X)\sqcup\{\star\}$ and \[S(\Delta(X))^\star:=\left\{u\star v\in S(\Delta(X)^\star)\,\big|\,u,v\in M(\Delta(X))\right\}.\]
For any $q\in S(\Delta(X))^\star$ and $u\in\bk\langle\Delta(X)\rangle$, define the {\bf $u$-word} $q|_u:=q|_{\star\mapsto u}$ to be the polynomial in $\bk\langle\Delta(X)\rangle$ obtained by replacing the symbol $\star$ in $q$ with $u$ and then expanding the result by linearity.
\end{defn}

\begin{defn}
Let $\oline{f}$ denote the {\bf leading} monomial word of $f\in\bk\langle\Delta(X)\rangle$ with respect to $\prec$.
Let lc$(f)$ be the leading coefficient of $f$, and denote
\[f^\natural:=\mbox{lc}(f)^{-1}f,\]
when $f\neq0$.
Then $f^\natural$ is a monic polynomial in $\bk\langle\Delta(X)\rangle$. In particular, $\oline{q|_u}=q|_{\oline{u}}$ for a $u$-word $q|_u$.
\mlabel{de:lead}
\end{defn}

Next we list some lemmas for words in $S(\Delta(X))$ in analog to \cite{QC}.
\begin{lemma}\label{lead}
Let $u_1,\dots,u_r\in \Delta(X)$ and $i\geq0$.
\begin{enumerate}
\item If weight $\lambda\neq0$, then the leading term of $d_X^i(u_1\cdots u_r)$ is  $\oline{d_X^i(u_1\cdots u_r)}=d_X^i(u_1)\cdots d_X^i(u_r)$ and its coefficient in $d_X^i(u_1\cdots u_r)$ is $\lambda^{(r-1)i}$.
\mlabel{lead1}
\item If weight $\lambda=0$, then the leading term of $d_X^i(u_1\cdots u_r)$ is $\oline{d_X^i(u_1\cdots u_r)}=d_X^i(u_1)u_2\cdots u_r$ and its coefficient in $d_X^i(u_1\cdots u_r)$ is 1. 	 \mlabel{lead2}
\end{enumerate}	
In particular, if $u,v\in S(\Delta(X))$ and $u\prec v$, then $d_X(u)\prec d_X(v)$.
\end{lemma}

\begin{lemma}\label{mono}
For any $u,v\in S(\Delta(X))$ and $q\in S(\Delta(X))^\star$,  if $u\prec v$, then $\oline{q|_u}\prec\oline{q|_v}$. Consequently, the deg-lex order $\prec$ on $M(\Delta(X))$ is a monomial order.
\end{lemma}

In order to define the Gr\"{o}bner-Shirshov bases in $\bk\langle\Delta(X)\rangle$, we need the concepts of intersection compositions and inclusion compositions of polynomials to involve the derivation $d_X$. In particular, by Eq. (\mref{diff}) we can use the notation of $u$-words in Definition \mref{def:uword} to define the inclusion compositions, analogous to those defined in \mcite{QC} via normal $u$-words.

\begin{defn}
We continue with the notation in Definition \mref{de:lead}.	
\begin{enumerate}
\item
For monic polynomials $f,g$ in $\bk\langle\Delta(X)\rangle$, if there exist $u,v,w_{i,j}\in M(\Delta(X))$ with $i,j\geq0$ such that $w_{i,j}=\oline{d_X^i(f)}u=v\oline{d_X^j(g)}$ and $|w_{i,j}|<|\oline{f}|+|\oline{g}|$, then we define
\[(f,g)_{w_{i,j}}^{u,v}:=d_X^i(f)^\natural u-vd_X^j(g)^\natural,\]
and call it an {\bf intersection composition} of $f$ and $g$ with respect to $w_{i,j}$.
\item
If there exist $w_{i,j}\in M(\Delta(X)),q\in S(\Delta(X))^\star$ with $i,j\geq0$ such that $w_{i,j}=\oline{d_X^i(f)}=q|_{\oline{d_X^j(g)}}$,
then we define
\[(f,g)_{w_{i,j}}^q:=d_X^i(f)^\natural-q|_{d_X^j(g)^\natural},\]
and call it an {\bf inclusion composition} of $f$ and $g$ with respect to $w_{i,j}$.
\end{enumerate}
\end{defn}

\begin{defn}\mlabel{gbda}
	Let $S$ be a set of monic polynomials in $\bk\langle\Delta(X)\rangle$ and $w\in M(\Delta(X))$. 	
	\begin{enumerate}
\item An element $u\in\bk\langle \Delta(X)\rangle$ is called {\bf trivial modulo $(S,w)$} and denoted
     $$u\equiv 0 \mod (S,w),$$
if $u=\sum_r c_rq_r|_{d_X^{k_r}(s_r)}$ with $c_r\in\bk$, $q_r\in S(\Delta(X))^\star$, $s_r\in S$ and $k_r\geq 0$ such that		
		$q_r|_{\oline{d_X^{k_r}(s_r)}}\prec w$ for any $r$.
For $u, v\in \bk\langle \Delta(X)\rangle$, denote $u\equiv v \mod (S,w)$ if $u-v\equiv 0 \mod (S,w)$.
\mlabel{it:gbda1}		
		\item For $f,g\in\bk\langle \Delta(X)\rangle$ and $u,v,w_{i,j}\in M(\Delta(X)),q\in S(\Delta(X))^\star$ that give an intersection composition $(f,g)_{w_{i,j}}^{u,v}$ or an inclusion composition $(f,g)_{w_{i,j}}^q$, this composition is called {\bf trivial modulo} $(S,w_{i,j})$, if
		\[(f,g)_{w_{i,j}}^{u,v} \text{ or } (f,g)_{w_{i,j}}^q\equiv 0\mod(S,w_{i,j}).\]
		
		\item The set $S$ is called a {\bf (differential) Gr\"{o}bner-Shirshov basis} if for any $f,g\in S$, all compositions $(f,g)^{u,v}_{w_{i,j}}$ or $(f,g)_{w_{i,j}}^q$ are trivial modulo $(S,w_{i,j})$. Moreover, a Gr\"{o}bner-Shirshov basis $S$ in $\bk\langle \Delta(X)\rangle$ is called {\bf minimal} if $S$ does not have inclusion compositions. It is called {\bf reduced} if every $f\in S$ is a linear combination of monomials in
		\[M(\Delta(X))\backslash \big\{ q|_{\oline{d_X^j(g)}}\,\left|\,q\in S(\Delta(X))^\star, g\in S\backslash\{d_X^i(f)\,|\,i\geq0\},j\geq0\right.\big\}.\]
		More precisely, when $f\in S$ is expressed as a linear combination of the basis $M(\Delta(X))$, the basis elements can not come from $\big\{ q|_{\oline{d_X^j(g)}}\,\left|\,q\in S(\Delta(X))^\star, g\in S\backslash\{d_X^i(f)\,|\,i\geq0\},j\geq0\right.\big\}$.
	\end{enumerate}
\end{defn}

The following CD lemma in $\bk\langle\Delta(X)\rangle$ is the special case of \cite[Theorem 3.6]{QC} when their operator set $\Omega$ is empty.
\begin{lemma}\mlabel{cdld}
	{\em(Composition-diamond lemma for differential algebras)} \ Let $S$ be
	a monic subset of  $\bk\langle\Delta(X)\rangle$, $\DI(S)$ be the differential ideal of
	$\bk\langle\Delta(X)\rangle$ generated by $S$ and
	$\prec$ be the order on $M(\Delta(X))$ defined in Eq.~\eqref{eq:diffmon}. Then the following
	statements are equivalent.
	\begin{enumerate}
		\item The set $S $ is a Gr\"{o}bner-Shirshov basis in  $\bk\langle\Delta(X)\rangle$; \mlabel{it:cdld1}
		\item If $f\in\DI(S)\backslash\{0\}$, then
		$\bar{f}=q|_{\overline{d_X^i(s)}}$  for some $q \in
		S(\Delta(X))^\star$, $s\in S$ and $i\geq 0$;
\mlabel{it:cdld2}
		\item \mlabel{it:cdld3}
The set of {\bf differential $S$-irreducible words}
$$\DIrr(S) := M(\Delta(X))\backslash \Big\{\left . q|_{\overline{d_X^i(s)}}\,\right |\, s\in S,\ q\in S(\Delta(X))^\star,\,i\geq0
\big\}$$
is a $\bk$-basis of $\bk\langle\Delta(X)\,|\,S\rangle:=\bk\langle\Delta(X)\rangle/\DI(S)$. In particular, $\bk\langle\Delta(X)\rangle=\bk\DIrr(S)\oplus\DI(S)$.

	\end{enumerate}

\end{lemma}

\subsubsection{Gr\"{o}bner-Shirshov bases for free commutative differential algebras}
Next we turn to the commutative case, beginning with reviewing the classical notion of Gr\"{o}bner bases~\cite{Buc}. See also~\cite[\S 2.2]{BC}.

Denote by $[X]$ the free commutative monoid on $X$. Any well-order $<$ on $X$ induces a monomial order $<$ on $[X]$ satisfying
\[1<u\text{ and }u<v\,\Rightarrow\,uw<vw\text{ for any }u,v,w\in [X]\backslash \{1\}.\]
Let $\oline{f}\in [X]$ be the leading term
of $f\in\bk[X]$ with respect to this order $<$. A {\bf Gr\"{o}bner(-Shirshov) basis} in $\bk[X]$ can also be defined by the triviality of its (intersection) compositions.
Then an argument similar to the proof of the CD Lemma \ref{cdla} establishes
the following classical result of Buchberger \cite{Buc}.

\begin{lemma}
	{\em (Composition-Diamond lemma for commutative algebras)} \ Let $S$ be a set of monic polynomials in $\bk[X]$, $\ideal(S)$ be the ideal of $\bk[X]$ generated by $S$. Let $<$ be a monomial order on $[X]$. Then the following statements are equivalent:	
	\begin{enumerate}
		\item the set $S$ is a Gr\"{o}bner(-Shirshov) basis in $\bk[X]$;		
		\item if $f\in \ideal(S)\backslash\{0\}$, then $\oline{f}=\oline{s}u$	for some $u\in[X]$ and $s\in S$;
		\item the set of {\bf $S$-irreducible words} $\Irr(S): = [X]\backslash\left\{\oline{s}u\,\big|\,u\in[X],s\in S\right\}$
		is a $\bk$-basis of  $\bk[X\,|\,S]=\bk[X]/\ideal(S)$. In particular, $\bk[X]=\bk\Irr(S)\oplus\ideal(S)$.
	\end{enumerate}
	\mlabel{ccdla}
\end{lemma}

Next for a well-ordered set $X$, let $[\Delta(X)]$ be the free commutative monoid \on $\Delta(X)$, endowed with the monomial order $\prec$ characterized by
\begin{equation}
u\prec v\,\Rightarrow\,uw\prec vw \text{ for any }u,v,w\in[\Delta(X)].
\mlabel{cmono}
\end{equation}
Starting with the order $\prec$ on $\Delta(X)$ as before, for any $\fraku:=u_1\cdots u_p$ and $\frakv:=v_1\cdots v_q\in[\Delta(X)]$ such that $u_1\succeq\cdots\succeq u_p,\,v_1\succeq\cdots\succeq v_q$, set
\begin{equation}
\fraku\prec\frakv\,\Leftrightarrow\,p<q\text{ or }p=q,\,u_1=v_1,\cdots,u_{k-1}=v_{k-1},\text{ but }u_k\prec v_k\text{ for some }k\leq p.
\mlabel{eq:cmon}
\end{equation}
Define the {\bf degree} $|u|$ of $u\in [\Delta(X)]$ to be the number of letters of $\Delta(X)$ in $u$, counting multiplicity.

Note that $[\Delta(X)]$ is also the monomial $\bk$-basis of $\bk[\Delta(X)]$ consisting of
\[x_{i_1}^{(j_1)}\cdots x_{i_n}^{(j_n)}\]
for any $x_{i_1}^{(j_1)},\dots, x_{i_n}^{(j_n)}\in\Delta(X),n\geq0$, such that  $x_{i_1}^{(j_1)}\succeq \cdots\succeq  x_{i_n}^{(j_n)}$.

As in the noncommutative case, we give
\begin{defn}
	Let $\oline{f}$ denote the {\bf leading} monomial word of $f\in\bk[\Delta(X)]$ with respect to $\prec$.
	Let lc$(f)$ be the leading coefficient of $f$, and denote
	\[f^\natural:=\mbox{lc}(f)^{-1}f\]
	when $f\neq0$, so that $f^\natural$ is a monic polynomial in $\bk[\Delta(X)]$.
	\mlabel{de:clead}
\end{defn}

Now we determine the leading terms and coefficients of differential words.
\begin{lemma}\label{clead}
	Let $u_1,\dots,u_r\in \Delta(X)$ satisfy $u_1\succeq\cdots\succeq u_r$ and $i\geq0$.
	\begin{enumerate}
		\item If weight $\lambda\neq0$, then the leading term of $d_X^i(u_1\cdots u_r)$ is  $\oline{d_X^i(u_1\cdots u_r)}=d_X^i(u_1)\cdots d_X^i(u_r)$ and its coefficient in $d_X^i(u_1\cdots u_r)$ is $\lambda^{(r-1)i}$.
		\mlabel{clead1}
		\item If weight $\lambda=0$, then the leading term of $d_X^i(u_1\cdots u_r)$ is $\oline{d_X^i(u_1\cdots u_r)}=d_X^i(u_1)u_2\cdots u_r$ and its coefficient in $d_X^i(u_1\cdots u_r)$ is the multiplicity of the letter $u_1$ in $u_1\cdots u_r$. 	
		\mlabel{clead2}
	\end{enumerate}	
	In particular, if $u,v\in [\Delta(X)]\backslash\{1\}$ and $u\prec v$, then $d_X(u)\prec d_X(v)$.
\end{lemma}
\begin{proof}
	This is easy to verify by Eq.~\eqref{dm} and the definition of monomial order $\prec$ on $[\Delta(X)]$.
\end{proof}

As in the case of the classical Gr\"{o}bner bases in $\bk[X]$, there is only one kind of compositions for commutative differential algebras.

\begin{defn}
	We continue with the notation in Definition~\mref{de:clead}.	
	For monic polynomials $f,g$ in $\bk[\Delta(X)]$, if there exist $u,v,w_{i,j}\in [\Delta(X)]$ with $i,j\geq0$ such that $w_{i,j}=\oline{d_X^i(f)}u=\oline{d_X^j(g)}v$ and $|w_{i,j}|<|\oline{f}|+|\oline{g}|$, then
	\[[f,g]_{w_{i,j}}^{u,v}:=d_X^i(f)^\natural u-d_X^j(g)^\natural v\]
	is called a {\bf composition} of $f$ and $g$ with respect to $w_{i,j}$. Since $u$ and $v$ are uniquely determined by $w_{i,j}$, we will sometimes abbreviate $[f,g]_{w_{i,j}}^{u,v}$ as $[f,g]_{w_{i,j}}$.
\end{defn}

\begin{defn}
	Let $S$ be a set of monic polynomials in $\bk[\Delta(X)]$ and $w\in [\Delta(X)]$. 	
	\begin{enumerate}
		\item An element $u\in\bk[\Delta(X)]$ is called {\bf trivial congruent modulo} $(S,w)$, denoted by
		\[u\equiv 0\mod(S,w),\]
		if $u=\sum_r c_rd_X^{k_r}(s_r)u_r$ with $c_r\in\bk$, $s_r\in S$, $u_r\in [\Delta(X)]$ and $k_r\geq 0$ such that $\oline{d_X^{k_r}(s_r)}u_r\prec w$ for any $r$.
For $u,v\in\bk[\Delta(X)]$, we denote
$u\equiv v\mod(S,w)$ if $u-v\equiv 0 \mod (S,w)$.
		\mlabel{it:cgbda1}		
		\item For $f,g\in\bk[\Delta(X)]$ and $u,v,w_{i,j}\in [\Delta(X)]$, a composition $[f,g]_{w_{i,j}}^{u,v}$ is called {\bf trivial modulo} $(S,w_{i,j})$, if
		\[[f,g]_{w_{i,j}}^{u,v}\equiv 0\mod(S,w_{i,j}).\]
		
		\item The set $S$ is called a {\bf (differential) Gr\"{o}bner-Shirshov basis}, if for any $f,g\in S$, all compositions $[f,g]^{u,v}_{w_{i,j}}$ are trivial modulo $(S,{w_{i,j}})$.	
	\end{enumerate}
	\mlabel{cgbda}
\end{defn}

\begin{lemma}
	Let $<$ be a monomial order on $[\Delta(X)]$ and $S$ be a set of monic polynomials in $\bk[\Delta(X)]$. Then the following conditions on $S$ are equivalent:
	\begin{enumerate}
		\item $S$ is a Gr\"{o}bner-Shirshov basis in $\bk[\Delta(X)]$.
		\mlabel{ckey0a}
		\item For any $s_1,s_2\in S$ and $w_{i,j}\in [\Delta(X)]$ such that $w_{i,j}=\oline{d_X^i(s_1)}u=\oline{d_X^j(s_2)}v$ with $u,v\in [\Delta(X)]$ and $i,j\geq0$, we have
$[s_1,s_2]^{u,v}_{w_{i,j}}\equiv 0\mod(S,w_{i,j})$.
		\mlabel{ckey0b}
	\end{enumerate}
	\mlabel{ckey0}
\end{lemma}
\begin{proof}
	First (\mref{ckey0b}) implies (\mref{ckey0a}) by Definition~\mref{cgbda}. Now to show that (\mref{ckey0a}) implies (\mref{ckey0b}), there are two cases to check depending on whether $\oline{d_X^i(s_1)}$ and $\oline{d_X^j(s_2)}$ are separated or not as subwords in $w_{i,j}$.
	\smallskip
	
	\noindent
	(a) Suppose that $\oline{d_X^i(s_1)}$ and $\oline{d_X^j(s_2)}$ are separated as subwords in $w_{i,j}$. Then there exists $w'\in [\Delta(X)]$ such that
	 \[w_{i,j}=\oline{d_X^i(s_1)}u=\oline{d_X^j(s_2)}v=\oline{d_X^i(s_1)d_X^j(s_2)}w',\]
	and thus
	\[\begin{split}
	d_X^i(s_1)^\natural u- d_X^j(s_2)^\natural v
	 &=d_X^i(s_1)^\natural\oline{d_X^j(s_2)}w'-d_X^j(s_2)^\natural \oline{d_X^i(s_1)}w'\\
	 &=-d_X^i(s_1)^\natural\left(d_X^j(s_2)^\natural-\oline{d_X^j(s_2)}\right)w'+d_X^j(s_2)^\natural\left(d_X^i(s_1)^\natural-\oline{d_X^i(s_1)}\right)w'.
	\end{split}\]
	Let
	\[d_X^i(s_1)^\natural-\oline{d_X^i(s_1)}=\sum_k c_k u_k
	\quad\text{and}\quad
	d_X^j(s_2)^\natural-\oline{d_X^j(s_2)}=\sum_l d_l v_l,\]
	with all $c_k,d_l\in\bk,\,u_k,v_l\in [\Delta(X)]$. Since
	 \[\oline{d_X^i(s_1)^\natural-\oline{d_X^i(s_1)}}\prec\oline{d_X^i(s_1)}, \quad \oline{d_X^j(s_2)^\natural-\oline{d_X^j(s_2)}}\prec\oline{d_X^j(s_2)},\]
	implying that
	\[\oline{d_X^i(s_1)}  v_lw'\prec\oline{d_X^i(s_1)d_X^j(s_2)}w'=w_{i,j}, \quad \oline{d_X^j(s_2)}u_kw'\prec\oline{d_X^j(s_2)d_X^i(s_1)}w'=w_{i,j},\]
	we have $d_X^i(s_1)^\natural u- d_X^j(s_2)^\natural v\equiv 0\mod(S,w_{i,j})$.
	
	\smallskip
	
	\noindent
	(b) Otherwise, $\oline{d_X^i(s_1)}$ and $\oline{d_X^j(s_2)}$ are overlapping  subwords in $w_{i,j}$. Then as $[\Delta(X)]$ is commutative, there exists a decomposition $w_{i,j}=w'w''$ for $w',w''\in [\Delta(X)]$ such that $w'=\oline{d_X^i(s_1)}u'=\oline{d_X^j(s_2)}v'$ for
	some $u',v'\in [\Delta(X)]$ satisfying $|w'|<|s_1|+|s_2|$. In particular, $u=u'w'',\,v=v'w''$.
	Since $S$ is a Gr\"{o}bner-Shirshov basis, by definition we have the composition
	\[[s_1,s_2]^{u',v'}_{w'}=d_X^i(s_1)^\natural u'-d_X^j(s_2)^\natural v'=\sum_r c_rd_X^{k_r}(t_r)u_r,\]
	with $c_r\in\bk$, $t_r\in S$, $u_r\in [\Delta(X)]$ and $k_r\geq 0$ such that $\oline{d_X^{k_r}(t_r)}u_r\prec w'$ for any $r$. By the property (\mref{cmono}) of $\prec$ on $[\Delta(X)]$, this means that
	\[d_X^i(s_1)^\natural u-d_X^j(s_2)^\natural v=\sum_r c_rd_X^{k_r}(t_r)u_rw'',\]
	with $\oline{d_X^{k_r}(t_r)}u_rw''\prec w'w''=w_{i,j}$ for any $r$. Hence, $d_X^i(s_1)^\natural u- d_X^j(s_2)^\natural v\equiv 0\mod(S,w)$.
\end{proof}

Now we are ready to give the following CD lemma for $\bk[\Delta(X)]$.
\begin{lemma}	\mlabel{ccdld}

	{\em(Composition-diamond lemma for commutative differential algebras)} \ Let $S$ be
	a monic subset of  $\bk[\Delta(X)]$, $\DI(S)$ be the differential ideal of
	$\bk[\Delta(X)]$ generated by $S$ and
	$\prec$ be a monomial order on $[\Delta(X)]$. Then the following
	statements are equivalent:
	\begin{enumerate}
		\item the set $S $ is a Gr\"{o}bner-Shirshov basis in  $\bk[\Delta(X)]$; \mlabel{it:ccdld1}
		\item if $f\in\DI(S)\backslash\{0\}$, then
		$\bar{f}=\overline{d_X^i(s)}u$  for some $s\in S$, $u\in[\Delta(X)]$ and $i\geq 0$;
		\mlabel{it:ccdld2}
		\item 		\mlabel{it:ccdld3}
the set of {\bf differential $S$-irreducible words}
		$$\DIrr(S):= [\Delta(X)]\backslash \left\{\left.\overline{d_X^i(s)}u\,\right|\, s\in S,\ u\in [\Delta(X)],\ i\geq0\right\}$$
		is a $\bk$-basis of $\bk[\Delta(X)\,|\,S]:=\bk[\Delta(X)]/\DI(S)$. In particular, $\bk[\Delta(X)]=\bk\DIrr(S)\oplus\DI(S)$.
	\end{enumerate}
\end{lemma}
\begin{proof} (\mref{it:ccdld1})$\Longrightarrow$(\mref{it:ccdld2}): For $f\in\DI(S)\setminus\{0\}$, by Eq.~\eqref{dix} we can write
	\[f=\sum_{i=1}^n a_id_X^{k_i}(s_i)^\natural u_i,\]
	where $a_i\in\bk$, $s_i\in S$ and $u_i\in [\Delta(X)]$. Let $w_i=\oline{d_X^{k_i}(s_i)} u_i$. By arranging their leading terms we can assume that $w_1=\cdots=w_m\succ w_{m+1}\succeq\cdots\succeq w_n$.
	
	Now one can prove the result by induction on $m$. If $m=1$, then $\oline{f}=w_1=\oline{d_X^{k_1}(s_1)}u_1$. If $m>1$, then by Lemma \mref{ckey0} we have
	\[d_X^{k_2}(s_2)^\natural u_2-d_X^{k_1}(s_1)^\natural u_1=\sum_r b_rd_X^{h_r}(t_r)v_r,\]
	where $b_r\in\bk,t_r\in S,v_r\in[\Delta(X)]$ such that  $\oline{d_X^{h_r}(t_r)}v_r\prec w_1$. Hence,
	\[f=\sum_{i=1}^n a_id_X^{k_i}(s_i)^\natural u_i=(a_1+a_2)d_X^{k_1}(s_1)^\natural u_1+a_2\sum_r b_rd_X^{h_r}(t_r)v_r+\sum_{i=3}^n a_id_X^{k_i}(s_i)^\natural u_i.\]
	If $m>2$ or $a_1+a_2\neq0$, then the result follows by induction on $m$. If $m=2$ and $a_1+a_2=0$, then it follows by induction on $ w_1$ as all the monomials left are smaller than $w_1$ under $\prec$.
	\smallskip
	
	\noindent
	 (\mref{it:ccdld2})$\Longrightarrow$(\mref{it:ccdld3}):
For any $f\in\bk[\Delta(X)]$, we first show that $f$ can be expressed as a linear combination of elements in $\DIrr(S)$ modulo $\DI(S)$ by induction on $\oline{f}$ under $\prec$ with the initial step that $\oline{f}=1$. Indeed, since $[\Delta(X)]=\DIrr(S)\sqcup \left\{\left.\overline{d_X^i(s)}u\,\right|\, s\in S,\ u\in [\Delta(X)],\ i\geq0\right\}$,
if $\oline{f}=\overline{d_X^i(s)}u$ for some $s\in S,\ u\in [\Delta(X)],\ i\geq0$, then
$\oline{f-\mbox{lc}(f)d_X^i(s)^\natural u}\prec\oline{f}$ and
we can apply the induction hypothesis to $f-\mbox{lc}(f)d_X^i(s)^\natural u$ to obtain the following form of expansion
\[f=\sum_{u_p\in\DIrr(S),\,u_p\preceq\oline{f}}a_pu_p+\sum_{s_r\in S,\,\oline{d_X^{h_r}(s_r)}v_r\preceq\oline{f}}b_rd_X^{h_r}(s_r)v_r,\]
where $a_p,b_r\in\bk$. Otherwise, such kind of expansion can be obtained by applying the induction hypothesis to $f-\mbox{lc}(f)\oline{f}$ as $\oline{f-\mbox{lc}(f)\oline{f}}\prec\oline{f}$ if $\oline{f}\in\DIrr(S)$.
Now suppose that $\sum_{p=1}^na_pu_p=0$ in $\bk[\Delta(X)\,|\,S]$ with each $a_p\in\bk\setminus\{0\}$ and $u_p\in\DIrr(S)$. Then $g=\sum_{p=1}^na_pu_p$ is in $\DI(S)$. By Item~(\mref{it:ccdld2}), we have
	$\oline{g}=u_p\notin\DIrr(S)$ for some $p$, which is a contradiction. Hence, elements in $\DIrr(S)$ are $\bk$-linear independent modulo $\DI(S)$. Thus $\DIrr(S)$ is a $\bk$-basis of $\bk[\Delta(X)\,|\,S]$.
	\smallskip
	
	\noindent
	 (\mref{it:ccdld3})$\Longrightarrow$(\mref{it:ccdld1}): For any composition $[f,g]_{w_{i,j}}^{u,v}$ in $S$, by
	$[f,g]_{w_{i,j}}^{u,v}$ being in $\DI(S)$ and Item~(\mref{it:ccdld3}), we have
	 \[[f,g]_{w_{i,j}}^{u,v}=\sum_rb_rd_X^{h_r}(t_r)v_r,\]
	where $b_r\in\bk,t_r\in S,v_r\in [\Delta(X)]$ such that  $\oline{d_X^{h_r}(t_r)}v_r\preceq\oline{[f,g]_{w_{i,j}}^{u,v}}\prec w_{i,j}$. This implies Item~(\mref{it:ccdld1}) as $[f,g]_{w_{i,j}}^{u,v}\equiv 0\mod(S,w_{i,j})$.
\end{proof}

\section{Construction of free differential algebras on algebras}
\mlabel{sec:bases}
In this section, we apply the CD Lemma \mref{cdld} and Lemma \mref{ccdld} to obtain \gsbs for free differential algebras on algebras and free commutative differential algebras on commutative algebras, and to provide canonical bases for these free differential algebras.

\subsection{Gr\"obner-Shirshov bases of free differential algebras on algebras}
\mlabel{ss:gsfda}
Recall the following natural algebra embedding
\[
\hat{i}_X:\bk\langle X\rangle\rar\bk\langle\Delta(X)\rangle,\,x\mapsto x^{(0)},x\in X,\]
in Proposition~\ref{pp:atda}. We further define the algebra embeddings
\begin{equation}
\hat{i}_X^{(n)}:\bk\langle X\rangle\rar\bk\langle\Delta(X)\rangle,\,x\mapsto x^{(n)},x\in X
\mlabel{eq:emb}
\end{equation}
for all $n\geq0$. In particular, $\hat{i}_X^{(0)}=\hat{i}_X$.

Now we are in the position to state our first result on the free differential algebra on an algebra, addressing Question~\mref{qu:basis}.(\mref{it:basis1}).

\begin{theorem}\label{mtgs}
Let $A$ be an algebra and let $A=\bk\langle X\rangle/I_A$ be a presentation of $A$ by generating set $X$ and ideal $I_A$. Let $S$ be a Gr\"obner-Shirshov basis of $I_A$. Let $\hat{S}:=\hat{i}_X(S)$. Let $\freed{A}$ be the free differential algebra on $A$ with the presentation $\freed{A}=\bfk\langle \Delta(X)\rangle/\DI(\hat{S})$ in Proposition~\mref{pp:atda}, where $\DI(\hat{S})$ is the differential ideal of $\bfk\langle \Delta(X)\rangle$ generated by $\hat{S}$. Then $\hat{S}$ is a Gr\"{o}bner-Shirshov basis of $\DI(\hat{S})$. Furthermore, if $S$ is minimal $($resp. reduced$)$, then $\hat{S}$ is also minimal $($resp. reduced$)$.
\mlabel{thm:gsfda}
\end{theorem}

\begin{proof}
We distinguish the two cases of weight $\lambda\neq 0$ and $\lambda=0$.

When $\lambda\neq 0$, the ambiguities of all possible compositions in $\hat{S}$ are listed as below by Lemma~\ref{lead}.(\ref{lead1}):
\begin{enumerate}
	\item $w_{n,n}=\oline{d_X^n(s)}u=v\oline{d_X^n(t)}$ for some $s,t\in\hat{S},\,u,v\in M(\Delta(X))$ and $n\geq0$  such that $|\oline{s}|+|\oline{t}|>|w_{n,n}|$.
	\mlabel{it:lcase1}
	\item $w_{n,n}=\oline{d_X^n(s)}=q|_{\oline{d_X^n(t)}}$ for some $s,t\in\hat{S},\,q\in S(\Delta(X))^\star$ and $n\geq0$.
\mlabel{it:lcase2}
\end{enumerate}

In case (\mref{it:lcase1}), by Lemma \ref{lead} (\ref{lead1}) we have $u=\hat{i}_X^{(n)}(u')$ and $v=\hat{i}_X^{(n)}(v')$ for some $u', v'\in M(X)$ that satisfy $w_{n,n}=\oline{d_X^n(s)}u=v\oline{d_X^n(t)}$. Let $u_0=\hat{i}_X(u')$ and $v_0=\hat{i}_X(v')$. Then we have
\[\begin{split}
(s,t)^{u,v}_{w_{n,n}}&=d_X^n(s)^\natural u-vd_X^n(t)^\natural\\
&=\left(d_X^n(s)^\natural u-\lambda^{-(|w_{n,n}|-1)n} d_X^n(s u_0)\right)
-\left(v d_X^n(t)^\natural
-\lambda^{-(|w_{n,n}|-1)n} d_X^n(v_0t)\right)
+\lambda^{-(|w_{n,n}|-1)n} d_X^n(su_0-v_0t).
\end{split}\]
For the first term, we have
\[d_X^n(s)^\natural u-\lambda^{-(|w_{n,n}|-1)n} d_X^n(s u_0)=\lambda^{-(|s|-1)n}\left(d_X^n(s)u-\lambda^{-|u|n} d_X^n(su_0)\right)=\sum_{j=0}^n\sum_{k_j} c_{j,\,k_j}q_{j,\,k_j}|_{d_X^j(s)},\,c_{j,\,k_j}\in\bk,\]
so that
$\oline{q_{j,\,k_j}|_{d_X^j(s)}}\prec w_{n,n}$, according to Eq.~\eqref{dm} and Lemma \ref{lead}. Similarly, for the second term, we have
$vd_X^n(t)
-\lambda^{-|v|n}d_X^n(v_0t)$ is also trivial modulo $(\hat{S},w_{n,n})$.
Meanwhile, $S$ is a Gr\"{o}bner-Shirshov basis in $\bk\langle X\rangle$, and it makes $(s,t)^{u_0,v_0}_{w_{0,0}}=su_0-v_0t\equiv 0\mod (\hat{S},w_{0,0})$ for $w_{0,0}=\oline{s}u_0=v_0\oline{t}$. Thus we also have $d_X^n(su_0-v_0t)
\equiv 0\mod (\hat{S},w_{n,n})$, as $w_{n,n}=\oline{d_X^n(w_{0,0})}$.
Putting together, we obtain
$$ (s,t)^{u,v}_{w_{n,n}}\equiv 0\mod (\hat{S},w_{n,n}).$$

For case (\mref{it:lcase2}), again by Lemma~\ref{lead}.(\mref{lead1}) we have $a=\hat{i}_X^{(n)}(a')$ and $b=\hat{i}_X^{(n)}(b')$ for some $a', b'\in M(X)$ to satisfy $q=a\star b$ and $w_n=\oline{d_X^n(s)}=q|_{\oline{d_X^n(t)}}=a\oline{d_X^n(t)}b$.
Let $a_0=\hat{i}_X(a')$, $b_0=\hat{i}_X(b')$ and $q_0=a_0\star b_0$. First note that
\[d_X^n(a_0tb_0)^\natural-ad_X^n(t)^\natural b=\lambda^{-(|w_{n,n}|-1)n}\left(d_X^n(a_0tb_0)-\lambda^{(|a|+|b|)n}ad_X^n(t)b\right)=\sum_{j=0}^n\sum_{k_j} c_{j,\,k_j}q_{j,\,k_j}|_{d_X^j(t)},\,c_{j,\,k_j}\in\bk,\]
such that
$\oline{q_{j,\,k_j}|_{d_X^j(t)}}\prec w_{n,n}$ by Eq.~\eqref{dm} and Lemma \ref{lead}. Also, the Gr\"{o}bner-Shirshov basis $S$ in $\bk\langle X\rangle$ guarantees $(s,t)^{q_0}_{w_{0,0}}=s-a_0tb_0\equiv 0\mod (\hat{S},w_{0,0})$ for $w_{0,0}=\oline{s}=a_0\oline{t}b_0$. Thus we also have $d_X^n(s-a_0tb_0)
\equiv 0\mod (\hat{S},w_{n,n})$, as $w_{n,n}=\oline{d_X^n(w_{0,0})}$. Hence, we obtain
\[\begin{split}
(s,t)^q_{w_{n,n}}&=d_X^n(s)^\natural-q|_{d_X^n(t)^\natural}=\lambda^{-(|w_{n,n}|-1)n}d_X^n(s)-ad_X^n(t)^\natural b\\
&=\lambda^{-(|w_{n,n}|-1)n}d_X^n(s-a_0tb_0)+\left(d_X^n(a_0tb_0)^\natural-ad_X^n(t)^\natural b\right)\equiv 0\mod (\hat{S},w_{n,n}).
\end{split}\]

\medskip
When $\lambda=0$, the ambiguities of all possible compositions come from the following cases by Lemma \ref{lead} (\mref{lead2}):

\begin{enumerate}
\item $w_{n,n}=\oline{d_X^n(s)}u=d_X^n(v)v'\oline{t}$ for some $s,t\in\hat{S},\,u\in \hat{i}_X(M(X)),\,v\in\hat{i}_X(X)$ and $v'\in M(\Delta(X))$ with $n\geq0$.
\mlabel{it:0case1}
\item $w_{n,n}=\oline{d_X^n(s)}=\oline{d_X^n(t)}u$ for some $s,t\in\hat{S},\,u\in \hat{i}_X(M(X))$ and $n\geq0$.
\mlabel{it:0case2}
\end{enumerate}

In case (\mref{it:0case1}), any ambiguity $w_{n,n}=\oline{d_X^n(s)}u=d_X^n(v)v'\oline{t},\,n>0$, is equivalent to $w_{0,0}=\oline{s}u=vv'\oline{t}$. In fact, we can prove that $(s,t)^{u,\,d_X^n(v)v'}_{w_{n,n}}=d_X^n(s)u-d_X^n(v)v't\equiv0\mod(\hat{S},w_{n,n}),\,n\geq0$ by induction on $n$. When $n=0$, it is clear as $S$ is a Gr\"{o}bner-Shirshov basis in $\bk\langle X\rangle$. Now for any $n\geq1$, by applying Eq.~\eqref{diff} when $\lambda=0$, the induction hypothesis, and the inequalities
$\oline{d_X^{n-1}(s)d_X(u)},\,\oline{d_X^{n-1}(v)d_X(v't)}\prec w_{n,n}$
by the definition of monomial order $\prec$ on $M(\Delta(X))$,
we obtain
\[\begin{split}
(s,t)^{u,\,d_X^n(v)v'}_{w_{n,n}}&=d_X^n(s)u-d_X^n(v)v't\\
&=d_X\left(d_X^{n-1}(s)u-d_X^{n-1}(v)v't\right) -d_X^{n-1}(s)d_X(u)+d_X^{n-1}(v)d_X(v't)\\
&\equiv0\mod(\hat{S},w_{n,n}).
\end{split}\]

For case (\mref{it:0case2}),  any ambiguity $w_{n,n}=\oline{d_X^n(s)}=\oline{d_X^n(t)}u,\,n>0$, is equivalent to $w_{0,0}=\oline{s}=\oline{t}u$. Since $\hat{S}=\hat{i}_X(S)$ and $S$ is a Gr\"{o}bner-Shirshov basis in $\bk\langle X\rangle$, we have $(s,t)^q_{w_{0,0}}=s-tu\equiv 0\mod (\hat{S},w_{0,0})$ for $q=\star u$.
Therefore,
\[\begin{split}
(s,t)^q_{w_{n,n}}&=d_X^n(s)-d_X^n(t)u=d_X^n(s-tu)+d_X^n(tu)-d_X^n(t)u\\
&\equiv d_X^n(tu)-d_X^n(t)u=\sum_{j=1}^n{n\choose j}d_X^{n-j}(t)d_X^j(u)\\
&\equiv 0\mod (\hat{S},w_{n,n}),
\end{split}\]
since $\oline{d_X^{n-j}(t)d_X^j(u)}\prec \oline{d_X^n(t)}u=w_{n,n}$ for any $0<j\leq n$ by comparing these monomial words of the same degree lexicographically.

To summarize, any inclusion composition in $\hat{S}$ has the form $(s,t)^q_{w_{n,n}},\,n\geq0$, trivial modulo $(\hat{S},w_{n,n})$, where especially $(s,t)^q_{w_{0,0}}$ must be the image of a trivial inclusion composition in $S$ under $\hat{i}_X$. Hence, if $S$ is minimal, that is, $S$ has no inclusion composition, so is $\hat{S}$.  On the other hand, $S$ being reduced implies that $\hat{S}$ is reduced as $\hat{S}=\hat{i}_X(S)$ and $\hat{i}_X$ is an embedding.
\end{proof}

\subsection{Linear basis of the free differential algebra on an algebra}
\mlabel{ss:basisfda}

By Theorem~\mref{thm:gsfda} and the CD Lemma~\ref{cdld}, the set $\DIrr(\hat{S})$ is a $\bk$-linear basis of $\freed{A}$. We now give an explicit description of this set.
More precisely, suppose that for the presentation $A=\bk\langle X\rangle/I_A$, $I_A$ has a \gsb $S_A$ which provides $A$ with a canonical linear basis. We will use the derived \gsb of the differential ideal $\DI(S_A)$ to provide the free differential algebra $\freed{A}=\free{X}/\DI(S_A)$ with a canonical linear basis. To give a precise statement, we start with some notations.

For the \gsb $S_A$ of $I_A$, let \[\oline{S_A}:=\{\oline{s}\in M(X)\,|\, s\in S_A\}\]
be the set of leading terms from $S_A$.
Also, for
$\fraku,\frakv\in M(\Delta(X))$, we denote $\frakv\mid\fraku$ if $\fraku=\fraka\frakv\frakb$ for some $\fraka, \frakb\in M(\Delta(X))$.
Otherwise, we write $\frakv\nmid\fraku$.

Recall the algebra embeddings defined in (\mref{eq:emb}),
\[\hat{i}_X^{(n)}:\bk\langle X\rangle\rar\bk\langle\Delta(X)\rangle,\,x\mapsto x^{(n)},x\in X,\]
for all $n\geq0$, and $\hat{i}_X=\hat{i}_X^{(0)}$. Given any polynomial $f\in\bk\langle X\rangle$, we further write
\[f^{(0)}:=\hat{i}_X(f)\in \bk\langle\Delta(X)\rangle.\]
For $\frakx=x_1\cdots x_k\in S(X)$ with $x_1,\cdots,x_k\in X$ and $n\geq 0$, denote
\begin{equation}
\frakx^{[n]}:=\begin{cases}
x_1^{(n)}\cdots x_k^{(n)},&\text{if }\lambda\neq0,\\
x_1^{(n)}x_2^{(0)}\cdots x_k^{(0)},&\text{if }\lambda=0.
\end{cases}
\mlabel{eq:diffpow}
\end{equation}

Now we have shown in Theorem~\mref{thm:gsfda} that for a \gsb $S_A$ of $I_A$, the embedding $\hat{S}_A=\hat{i}_X(S_A)=\left\{s^{(0)}\,\big|\, s\in S_A\right\}$ is a (differential) \gsb of the differential ideal $\DI(S_A)$.
\begin{prop}\mlabel{prop:fdab}
Let $A$ be an algebra with a presentation $\bk\langle X\rangle/I_A$ and let $S_A$ be a \gsb of the ideal $I_A$ in $\bk\langle X\rangle$.
The set of differential $\hat{S}_A$-irreducible elements
\begin{equation}
\DIrr(\hat{S}_A)=\left\{\frakx\in M(\Delta(X))\,\big|\, \oline{s}^{[n]}\nmid \frakx\text{ for any }\oline{s}\in \oline{S_A}, n\geq 0\right\},
\mlabel{eq:diffirr}
\end{equation}
giving a $\bfk$-linear basis of the free differential algebra $\freed{A}$ on $A$.
\end{prop}

\begin{proof}
First for any $s\in S_A$, let
$\oline{s}\in \oline{S_A}$. Since the deg-lex order $\prec$ on $M(\Delta(X))$ is a monomial order by Lemma \ref{mono}, we have $\overline{d_X^n(s^{(0)})}=\overline{d_X^n(\oline{s}^{(0)})} ,\,n\geq0$. Then
 by Lemma \ref{lead}, $\overline{d_X^n(s^{(0)})}$
is exactly $\oline{s}^{[n]}$ defined in Eq.~(\mref{eq:diffpow}).

On the other hand, according to Theorem \ref{mtgs} the set $\hat{S}_A$
is a \gsb of the differential ideal $\DI(\hat{S}_A)$ in $\free{X}$ for arbitrary weight $\lambda$. Then by
the CD Lemma \ref{cdld}, the set of differential $\hat{S}_A$-irreducible elements
\[\DIrr(\hat{S}_A)=\left\{\frakx\in M(\Delta(X)) \,\Big|\,  \frakx \neq
\fraka\overline{d_X^n(s^{(0)})}\frakb\text{ for any }s\in S_A,\,\fraka,\frakb\in M(\Delta(X)),\,n\geq 0\right\}\]
is a linear basis of the free differential algebra $\freed{A}=\free{X}/\DI(\hat{S}_A)$.
Now  $\oline{s}^{[n]}=\overline{d_X^n(s^{(0)})}$
and $\frakx \neq
\fraka\overline{d_X^n(s^{(0)})}\frakb$ for any $\fraka,\frakb\in M(\Delta(X))$ mean  that $\oline{s}^{[n]}\nmid \frakx$. Thus $\DIrr(\hat{S}_A)$ can be rewritten as in Eq.~(\mref{eq:diffirr}).
\end{proof}

\subsection{Gr\"obner-Shirshov bases of free commutative differential algebras on commutative algebras}\mlabel{ss:gsfcda}

As we now see, in contrast to Theorem~\ref{mtgs}, its commutative version only holds for weight $\lambda\neq0$. If $\lambda=0$, the extension $\hat{S}$ of Gr\"{o}bner-Shirshov basis $S$ in  $\bk[X]$ may fail to be a \gsbs in  $\bk[\Delta(X)]$.

Define the following commutative algebra embeddings as in Eq.~(\mref{eq:emb}),
\[\hat{i}_X^{(n)}:\bk[X]\rar\bk[\Delta(X)],\,x\mapsto x^{(n)},x\in X\]
for all $n\geq0$, and $\hat{i}_X:=\hat{i}_X^{(0)}$.
\begin{theorem}
Let $A=\bk[X]/I_A$ be a commutative algebra with a generating set $X$ and a defining ideal $I_A$. If $I_A$ is generated by a Gr\"{o}bner-Shirshov basis $S_A$ in $\bk[X]$, then
\[\hat{S}_A:=\hat{i}_X(S_A)\]
is a Gr\"{o}bner-Shirshov basis in $\bk[\Delta(X)]$ when $\lambda\neq0$, so that $\freecd{A}$ is isomorphic to $\bk[\Delta(X)]/\DI(\hat{S}_A)$ as commutative differential algebras.
\mlabel{mtgsc}
\end{theorem}

\begin{proof}
When $\lambda\neq 0$, one can easily modify the proof in Theorem \ref{mtgs} for the commutative case, with the ambiguities of all possible compositions in $\hat{S}_A$ as follows:
\[w_{n,n}=\oline{d_X^n(s)}u=\oline{d_X^n(t)}v\mbox{ for some }s,t\in\hat{S}_A,\,u,v\in[\Delta(X)]\mbox{ and }n\geq0\mbox{ such that }|\oline{s}|+|\oline{t}|>|w_{n,n}|,\]
according to Lemma \ref{clead} (\mref{clead1}). Then
it is similar to check that
$$ (s,t)^{u,v}_{w_{n,n}}=d_X^n(s)^\natural u-d_X^n(t)^\natural v\equiv 0\mod (\hat{S}_A,w_{n,n}).$$
Hence, $\hat{S}_A$
is a Gr\"{o}bner-Shirshov basis in $\bk[\Delta(X)]$ when $\lambda\neq0$.
\end{proof}

\begin{remark}\label{cgs0}
If weight $\lambda=0$, such an extension of Gr\"{o}bner-Shirshov bases might not be determined in general.
	
Indeed, when $\lambda=0$, the ambiguities of all possible compositions in $\hat{S}_A$ have the form
\[w_{m,n}=\oline{d_X^m(s)}u=\oline{d_X^n(t)}v\mbox{ for some }s,t\in\hat{S}_A,\,u,v\in [\Delta(X)]\mbox{ and }m,n\geq0\mbox{ such that }|\oline{s}|+|\oline{t}|>|w_{m,n}|,\]
by Lemma \ref{clead} (\mref{clead2}).
In this situation, we have more complicated compositions $[s,t]^{u,v}_{w_{m,n}}$ with different $m,n\geq0$, which do not appear in the other cases. Instead, it depends on the concrete algebra structure of $A$ to determine whether $[s,t]^{u,v}_{w_{m,n}}$ are trivial modulo $(\hat{S}_A,w_{m,n})$ or not. We will see that both possibilities can arise by examples in  Section \mref{ss:exam}.
\mlabel{rk:gswt0}
\end{remark}

As mentioned in Remark \ref{cgs0}, all possible compositions in $\hat{S}_A$ are of the form $[s,t]^{u,v}_{w_{m,n}}$ when $\lambda=0$. Thus we propose

\begin{conjecture}
	Let $A=\bk[X]/I_A$ be a commutative algebra with a generating set $X$ and a defining ideal $I_A$. If $I_A$ is generated by a Gr\"{o}bner-Shirshov basis $S_A$ in $\bk[X]$, let
	\[\hat{S}_A:=\hat{i}_X(S_A).\]
	Then the enlarged set $\hat{S}_A^+$ defined by
	 \[\hat{S}_A\cup\left\{[s,t]^{u,v}_{w_{m,n}}\,\big|\,s,t\in\hat{S}_A,\,u,v\in[\Delta(X)],\,m,n\geq0\mbox{ such that }|\oline{s}|+|\oline{t}|>|w_{m,n}|\right\}\]
	is a Gr\"{o}bner-Shirshov basis in $\bk[\Delta(X)]$ when $\lambda=0$, such that $\calc\cald_0(A)$ is isomorphic to $\bk[\Delta(X)]/\DI(\hat{S}_A^+)$ as commutative differential algebras.	
\end{conjecture}

\subsection{Linear basis of the free commutative differential algebra on a commutative algebra}
\mlabel{ss:basisfcda}

We now apply Theorem~\mref{mtgsc} to give a canonical basis of the free commutative differential algebra on a commutative algebra if the latter is given by a \gsb, when the weight is not zero. When the weight is zero, we provide an example where the statement of Theorem~\mref{mtgsc} still holds and leave further discussions to the next subsection.

For a well-ordered set $X=\{x_i\,|\,i\in I\}$, the free commutative monoid on $\Delta(X)$ can be expressed as
\[[\Delta(X)]=\left\{\prod_{i\in I}\prod_{k\geq0}(x_i^{(k)})^{a_{ik}}\,\bigg|\,a_{ik}\geq0,\sum_{i\in I,k\geq 0} a_{ik}<\infty\right\}.\]
Define a partial order $\mid$ on $[\Delta(X)]$ as follows.
For $\fraka=\prod_{i\in I}\prod_{k\geq0}(x_i^{(k)})^{a_{ik}},\,\frakb=\prod_{i\in I}\prod_{k\geq0}(x_i^{(k)})^{b_{ik}}\in[\Delta(X)]$, we say $\fraka$ {\bf divides} $\frakb$, denoted by $\fraka\mid\frakb$, if
\[a_{ik}\leq b_{ik}\text{ for all }i\in I,k\geq0.\]
Otherwise, we write $\fraka\nmid\frakb$. Note that $\fraka\mid\frakb$ is clearly equivalent to that $\frakb=\fraka\frakc$ for some $\frakc\in[\Delta(X)]$.

For $n\geq 0$ and $\frakx=x_1\cdots x_k\in [X]$ with $x_1\geq\cdots\geq x_k\in X$, we define
\begin{equation}
\frakx^{[n]}:=\begin{cases}
x_1^{(n)}\cdots x_k^{(n)},&\text{if }\lambda\neq0,\\
x_1^{(n)}x_2^{(0)}\cdots x_k^{(0)},&\text{if }\lambda=0.
\end{cases}
\mlabel{eq:cdiffpow}
\end{equation}
Further, we write $f^{(0)}:=\hat{i}_X(f)\in \bk[\Delta(X)]$
for any polynomial $f\in\bk[X]$.

\begin{prop}
Let $A$ be a commutative algebra with presentation $\bk[X]/I_A$ and $S_A$ be a Gr\"{o}bner$($-Shirshov$)$ basis of $I_A$ in $\bk[X]$. Let $\oline{S_A}:=\{\bar{s}\in[X]\,|\, s\in S_A\}$ be the set of leading terms from $S_A$, and $\hat{S}_A:=\hat{i}_X(S_A)=\left\{s^{(0)}\,\big|\, s\in S_A\right\}$.
When $\lambda\neq 0$, the set of differential $\hat{S}_A$-irreducible words
\begin{equation}
\DIrr(\hat{S}_A):=\left\{\frakx\in[\Delta(X)]\,\Big|\,\oline{s}^{[n]}\nmid\frakx\text{ for any }\oline{s}\in\oline{S_A},\,n\geq0 \right\},
\mlabel{eq:cdiffirr}
\end{equation}
giving a $\bk$-linear basis of $\freecd{A}$.
\mlabel{prop:fcdab}
\end{prop}
\begin{proof}
This proposition can be proven following the noncommutative case in Proposition \mref{prop:fdab}.
For any $s\in S_A,n\geq0$, let
$\oline{s}\in \oline{S_A}$.
By the property \eqref{cmono} of $\prec$ on $[\Delta(X)]$  and Lemma \ref{clead},
we know that $\overline{d_X^n(s^{(0)})}=\overline{d_X^n(\oline{s}^{(0)})}$ and $\overline{d_X^n(s)}$ is exactly $\oline{s}^{[n]}$ defined in (\mref{eq:cdiffpow}) for any $n\geq0$.

On the other hand, according to Theorem \ref{mtgsc}, the set $\hat{S}_A:=\hat{i}_X(S_A)$
is a \gsb of the differential ideal $\DI(S_A)$ in $\freecd{A}$ when $\lambda\neq0$. Hence, by CD Lemma \ref{ccdld}, the set of differential $\hat{S}_A$-irreducible elements
\[\DIrr(\hat{S}_A)=\left\{ u\in [\Delta(X)] \,\Big|\,  u \neq
\overline{d_X^n(s^{(0)})}v\text{ for any }s\in S_A,v\in[\Delta(X)],n\geq0\right\}\]
is a linear basis of  $\freecd{A}=\freec{X}/\DI(\hat{S}_A)$.
Then the fact that $\DIrr(\hat{S}_A)$ is also given by (\mref{eq:cdiffirr}) is due to the description of $\fraka\mid\frakb$ above.	
\end{proof}
	
As noted in Remark~\mref{cgs0}, in the case of weight $\lambda=0$, the question is open on whether or not a \gsb of a commutative algebra can be extended to a differential \gsb of the free commutative differential algebra on this commutative algebra. Here we will provide an example where the answer to the above question is positive, while it seems quite common that the answer to the question is negative from other examples in the next subsection.

Let $X=\{x,y\}$ and let $A$ be the algebra $A=\bk[x,y]/(x+y+1)$. With the order $x>y$, the set $\{x+y+1\}$ is a \gsb of the ideal $(x+y+1)$ and $\Irr(\{x+y+1\})=\{y^k\,|\,k\geq0\}$.

\begin{prop}\label{polyp}
	Let $A=\bk[x,y]/(x+y+1)$, and $X=\{x,y\}$ with $x>y$.
	\begin{enumerate}
		\item
		The free differential algebra $\freed{A}$ on $A$ of weight $\lambda$ is
		\[\bk\langle \Delta(X)\rangle/\left(x^{(m)}+y^{(m)}+\delta_{m,0}\,\big|\,m\geq0\right).\]
		The generating set $\hat{S}:=\left\{x^{(0)}+y^{(0)}+1\right\}$ of the differential ideal
		is a Gr\"{o}bner-Shirshov basis in $\bk\langle \Delta(X)\rangle$ for arbitrary weight $\lambda$. The set \[\DIrr(\hat{S})=M\left(\left\{y^{(k)}\,\bigg|\,k\geq0\right\}\right)\]
		is a $\bk$-basis of $\freed{A}$;
		\mlabel{it:polyp1}
		\item
		The free commutative differential algebra $\freecd{A}$ on $A$ of weight $\lambda$ (including $\lambda=0$) is
		 \[\bk[\Delta(X)]/\left(x^{(m)}+y^{(m)}+\delta_{m,0}\,\big|\,m\geq0\right).\]
		The generating set $\hat{S}:=\left\{x^{(0)}+y^{(0)}+1\right\}$ of the differential ideal
		is a Gr\"{o}bner-Shirshov basis
		in $\bk[\Delta(X)]$ for weight $\lambda$, and
		 \[\DIrr(\hat{S})=\left[\left\{y^{(m)}\,\bigg|\,m\geq0\right\}\right]\]
		is a $\bk$-basis of $\freecd{A}$.
	\end{enumerate}
\end{prop}
\begin{proof}
	We only prove the commutative case when $\lambda=0$. Other cases follows directly from Theorem \ref{mtgs}, Proposition \mref{prop:fdab}, Theorem \ref{mtgsc} and Proposition \mref{prop:fcdab}.
	
	According to Remark \ref{cgs0}, in order to show that $\hat{S}=\left\{x^{(0)}+y^{(0)}+1\right\}$ is a Gr\"{o}bner-Shirshov basis
	in $\bk[\Delta(X)]$ when $\lambda=0$, we only need to check that the following compositions in $\hat{S}$,
	 \[[x^{(0)}+y^{(0)}+1,x^{(0)}+y^{(0)}+1]_{w_{m,n}}=\left(x^{(m)}+y^{(m)}+\delta_{m,0}\right)x^{(n)}-\left(x^{(n)}+y^{(n)}+\delta_{n,0}\right)x^{(m)},\]
	are trivial modulo $(\hat{S},w_{m,n})$, where $w_{m,n}=x^{(m)}x^{(n)}$ for all $m,n\geq0$. Indeed,
	\[\begin{split}
	 [x^{(0)}&+y^{(0)}+1,x^{(0)}+y^{(0)}+1]_{w_{m,n}}=\left(y^{(m)}+\delta_{m,0}\right)x^{(n)}-\left(y^{(n)}+\delta_{n,0}\right)x^{(m)}\\
	 &=\left(y^{(m)}+\delta_{m,0}\right)\left(x^{(n)}+y^{(n)}+\delta_{n,0}\right)-\left(y^{(n)}+\delta_{n,0}\right)\left(x^{(m)}+y^{(m)}+\delta_{m,0}\right)\\
	&\equiv0\mod(\hat{S},w_{m,n}).
	\end{split}\]
	Furthermore, since $\oline{x^{(m)}+y^{(m)}+\delta_{m,0}}=x^{(m)}$ for all $m\geq0$, by CD Lemma \ref{ccdld}, the set
	 \[\DIrr(\hat{S})=\left[\left\{y^{(m)}\,\bigg|\,m\geq0\right\}\right]\]
	is a $\bk$-basis of $\freecd{A}=\bk[\Delta(X)]/\DI(\hat{S})$.
	
	In fact, it is easy to see that $\freecd{A}\cong\bk[y^{(k)}\,|\,k\geq0]$ as differential algebras.
\end{proof}

\subsection{Examples of free differential algebras on algebras}
\mlabel{ss:exam}

As demonstrations of the utility of the general results on free differential algebras on algebras in the previous sections, we give several concrete examples where the canonical basis can be made completely explicit.

\subsubsection{Free differential algebras on algebras with one generator}

Let $(R,d)$ be a differential algebra with one generator $u$ and let $A$ be the subalgebra of $R$ generated by $u$. Then $(R,d)$ is a differential algebra generated by $A$. Also $A$ is a quotient algebra of the free algebra on one generator and hence is of the form $\bk[x]/(f(x))$. It then follows that a free differential algebra on one generator is of the form $\freed{A}$ with $A=\bk[x]/(f(x))$ and a free commutative differential algebra on one generator is of the form $\freecd{A}$ with $A=\bfk[x]/(f(x))$.

\begin{prop}\mlabel{poly}
Let $A=\bk[x]/(f(x))$ with $f\in \bk[x]$ of degree $n>0$, and $X=\{x\}$.
\begin{enumerate}
\item
The free differential algebra $\freed{A}$ on $A$ is
\[\bk\langle \Delta(X)\rangle/\left(d_X^m(f(x^{(0)}))\,\big|\,m\geq0\right).\]
The generating set $\hat{S}:=\left\{f(x^{(0)})\right\}$ of the differential ideal
is a Gr\"{o}bner-Shirshov basis in $\bk\langle \Delta(X)\rangle$ for arbitrary weight $\lambda$. The set \[\DIrr(\hat{S})=\left\{\frakx\in M(\Delta(X))\,\Big|\,x^{(m)}(x^{(m(1-\delta_{\lambda,0}))})^{n-1}\nmid\frakx,\,m\geq0\right\}\]
is a $\bk$-basis of $\freed{A}$;
\mlabel{it:poly1}
\item\mlabel{it:poly2}
The free commutative differential algebra $\freecd{A}$ of weight $\lambda$ on $A$ is
\[\bk[\Delta(X)]/\left(d_X^m(f(x^{(0)}))\,\big|\,m\geq0\right).\]
Let $\lambda\neq 0$. Then the generating set $\hat{S}:=\left\{f(x^{(0)})\right\}$ of the differential ideal
is a Gr\"{o}bner-Shirshov basis
 in $\bk[\Delta(X)]$. Further, the $\bk$-basis $\DIrr(\hat{S})$ of $\freecd{A}$ can be written more explicitly as
 \[\left\{\prod\nolimits_{i\geq0}(x^{(i)})^{n_i}\,\bigg|\, n_i\in\{0,1,\dots,n-1\} \text{ and } n_i=0 \text{ for almost all } i\geq 0\right\}.\]
\end{enumerate}
\end{prop}
\begin{proof}
(\mref{it:poly1})
Taking $X=\{x\}$ in Proposition~\ref{pp:atda}, we have
\[\freed{A}=\bk\langle \Delta(X)\rangle\big/\DI(\hat{S})
\quad\text{and}\quad
\freecd{A}=\bk[ \Delta(X)]\big/\DI(\hat{S}),\]
with $\DI(\hat{S})=\left(d_X^m(f(x^{(0)}))\,\big|\,m\geq0\right)$. Clearly $S=\{f(x)\}$ is a Gr\"{o}bner-Shirshov basis in $\bk[x]$. Also by Theorem \ref{mtgs},
$\hat{S}=\left\{f(x^{(0)})\right\}$ is a Gr\"{o}bner-Shirshov basis in $\bk\langle \Delta(X)\rangle$ for arbitrary weight $\lambda$.
Since $f$ has the leading term $\oline{f}=x^n$, we have $\oline{f}^{[m]}=x^{(m)}(x^{(m(1-\delta_{\lambda,0}))})^{n-1},\,m\geq0$, and the $\bk$-basis $\DIrr(\hat{S})$ of $\freed{A}$ takes its desired form by Proposition \mref{prop:fdab}.
\smallskip

\noindent
\eqref{it:poly2}
The argument is similar for the commutative case $\freecd{A}$ when $\lambda\neq0$, according to Theorem~\ref{mtgsc} and Proposition \mref{prop:fcdab}, with its $\bk$-basis $\DIrr(\hat{S})$ easily written down as indicated.
\end{proof}	

\begin{remark}
However, when $\lambda=0$, $\hat{S}$ is not necessarily a Gr\"{o}bner-Shirshov basis in $\bk[\Delta(X)]$. Hence the corresponding statement in Proposition~\mref{poly} does not hold. In the following, we provide an example to show that whether the extension of Gr\"{o}bner-Shirshov bases works or not depends on the choice of monomial orders. See also Proposition~\ref{fcg}.
\end{remark}

Consider the algebra $A=\bk[x]/(x^2)$ of dual numbers and  $X=\{x\}$. The free commutative differential algebra $\freecd{A}$ on $A$ is
\[\bk[\Delta(X)]\big/\left(d_X^m((x^{(0)})^2)\,\big|\,m\geq0\right).\]
When $\lambda\neq0$, $\hat{S}=\left\{(x^{(0)})^2\right\}$
is a Gr\"{o}bner-Shirshov basis in $\bk[\Delta(X)]$ so that
\[\DIrr(\hat{S})=
\left\{\prod\nolimits_{i\geq0}(x^{(i)})^{n_i}\,\bigg|\,\sum\nolimits_{i\geq1}n_i<\infty\mbox{ with all }n_i=0,1\right\}\]
as a $\bk$-basis of $\freecd{A}$ with respect to the order $\prec$ given in Eq.~(\mref{eq:cmon}).

On the other hand, when $\lambda=0$, $\hat{S}=\left\{(x^{(0)})^2\right\}$ can not be a Gr\"{o}bner-Shirshov basis in $\bk[\Delta(X)]$ with respect to this order when $\lambda=0$.
Indeed, we find that the composition
\[[(x^{(0)})^2,(x^{(0)})^2]_{w_{2,1}} =2^{-1}\left(d_X^2((x^{(0)})^2)x^{(1)}-d_X((x^{(0)})^2)x^{(2)}\right) =(x^{(1)})^3,\]
with $w_{2,1}=x^{(2)}x^{(1)}x^{(0)}$, cannot be any linear combination of differential words in $\DI(\hat{S})$ with leading terms $\prec w_{2,1}$, namely,
$$d_X^2((x^{(0)})^3), d_X^2((x^{(0)})^2)x^{(0)}, x^{(2)}(x^{(0)})^2, x^{(1)}d_X((x^{(0)})^2), $$ $$3^{-1}d_X((x^{(0)})^3)=2^{-1}d_X((x^{(0)})^2)x^{(0)}=x^{(1)}(x^{(0)})^2, (x^{(0)})^3.
$$
Hence, it is nontrivial modulo $(\hat{S},w_{2,1})$ by Definition~\ref{cgbda}.(\mref{it:cgbda1}), and thus $\hat{S}$ is not a Gr\"{o}bner-Shirshov basis in $\bk[\Delta(X)]$ when $\lambda=0$. It can also be seen by CD Lemma~\mref{ccdld}, since now  $(x^{(1)})^3\in\DIrr(\hat{S})\cap\DI(\hat{S})$. This counter example shows that Proposition~\mref{poly}.(\mref{it:poly2}) does not hold when $\lambda=0$.

\medskip
Alternatively, due to the classical result of Levi in \cite[Theorem 1.1]{Le} with $p=2$, we obtain \[\DI\left((x^{(0)})^2\right)=\left(\sum\nolimits_{i=0}^m{m\choose i}x^{(m-i)}x^{(i)}\,\bigg|\,m\geq0\right)\]
and
\[\calc\cald_0(A)\cong\bk\left[x^{(r)}\,\big|\,r\geq0\right]\bigg/\DI\left((x^{(0)})^2\right)\]
have a $\bk$-basis consisting of the so-called {\bf $\gamma$-terms} and {\bf $\alpha$-terms} respectively.
Denote the monomial basis elements of $\bk[\Delta(X)]$ as
\[\frakx_\alpha:=\prod_{r\geq 0}(x^{(r)})^{a_r}\in [\Delta(X)],\]
where $\alpha=(a_0,a_1,\dots)$ with all $a_r\in\NN$ and almost all zero.
Then $\frakx_\alpha$ is called an $\alpha$-term if $a_r+a_{r+1}<2$ for all $r\geq0$,  equivalently for consecutive $x^{(r)}, x^{(r+1)}, r\geq0$ in $\frakx_\alpha$ at most one of them can appear in $\frakx_\alpha$ and can appear at most once.

Now we choose a lexicographic order $<_{\lex}$ on $[\Delta(X)]$ that is different from the previous $\prec$ in Eq.~(\mref{eq:cmon}). For any $\fraku:=u_1\cdots u_p$ and $\frakv:=v_1\cdots v_q\in [\Delta(X)]\backslash \{1\}$ with $u_1\preceq\cdots\preceq u_p,\,v_1\preceq\cdots\preceq v_q$ under the order (\mref{eq:diffmon}) on $\Delta(X)$, set
\[1<_{\lex}\fraku,\,\fraku<_{\lex}\frakv\,\Leftrightarrow\,u_1=v_1,\cdots,u_{k-1}=v_{k-1},\text{ but }u_k\prec v_k\text{ for some }k.\]
Then under this order $<_{\lex}$, the set $\DIrr(\hat{S})$ of differential $\hat{S}$-irreducible words is precisely the basis of $\calc\cald_0(A)$ consisting of $\alpha$-terms given by Levi, since now
\[\oline{d_X^m((x^{(0)})^2)}=\begin{cases}
(x^{(k)})^2,&m=2k,\\
x^{(k)}x^{(k+1)},&m=2k+1,
\end{cases} \quad k\geq 0.\]
By CD Lemma~\mref{ccdld}, the set $\hat{S}=\left\{(x^{(0)})^2\right\}$ is a Gr\"obner-Shirshov basis in $\bk[\Delta(X)]$
with respect to the lexicographic order $<_{\lex}$. Thus we find another supporting example of a weight $0$ commutative differential algebra for which the Gr\"obner-Shirshov basis of the generating algebra can be extended.

\subsubsection{Free differential algebras on some group algebras}

Denote the free (commutative) differential algebra on a finite (commutative) group algebra $\bk G$ more simply by
\[\freed{G}:=\freed{\bk G}\quad(\freecd{G}:=\freecd{\bk G}),\]
with differential operator $d_G$.

The following result shows that free commutative differential algebra on an algebra can be very degenerated.
\begin{prop}\label{fcg}
	For any finite commutative group $G$, the free commutative differential algebra of weight 0 on $\bk G$ is $(\bk G,0)$.
\end{prop}
\begin{proof}
First note that
		\[\bk G=\bk[g\,|\,g\in G]/(gh-g\cdot h,\,e-1\,|\,g,h\in G),\]
	where we use $\cdot$ to denote the multiplication of $G$, and $e$ is the unit in $G$.	
	Then by Proposition~\ref{pp:atda}, we have
	 \[\calc\cald_0(G)=\bk[\Delta(G)]/\left(d_G^n(g^{(0)}h^{(0)})-(g\cdot h)^{(n)},\,e^{(n)}-\delta_{n,0}\,|\,g,h\in G,\,n\geq0\right).\]
	Given any $g\in G$, there exists $r\geq 0$ such that $g^r=e$ in $\bk G$, since $|G|<\infty$. Then \[0=d_G(e^{(0)})=d_G((g^{(0)})^r)=rg^{(1)}(g^{(0)})^{r-1}\]
	by Eqs.~\eqref{diff} and \eqref{difu} with $\lambda=0$. Hence, $rg^{(1)}(g^{(0)})^r=rg^{(1)}=0$, which also means that $g^{(1)}=0$ as $\bk$ is of characteristic 0. Since $g$ is arbitrary, we have $d_G=0$, and thus $\calc\cald_0(G)=\bk G$.
\end{proof}

As a corollary of Proposition \ref{poly}, we also obtain the construction of free differential algebras for cyclic groups.

\begin{coro}\label{cy}
	For the cyclic group $C_n$ of order $n\geq2$, the free differential algebra $\freed{C_n}$ is isomorphic to
	\[\bk\langle \Delta(X) \rangle\big/\left(d_X^m((x^{(0)})^n)-\delta_{m,0}\,\big|\,m\geq0\right),\]
	with $X=\{x\}$, and $\hat{S}=\left\{(x^{(0)})^n-1\right\}$ is a Gr\"{o}bner-Shirshov basis in $\bk\langle \Delta(X)\rangle$ for arbitrary weight $\lambda$, such that the set
	\[\DIrr(\hat{S})=\left\{\frakx\in M(\Delta(X))\,\Big|\,x^{(m)}(x^{(m(1-\delta_{\lambda,0}))})^{n-1}\nmid\frakx,\,m\geq0\right\}\]
	is a $\bk$-basis of $\freed{C_n}$.
	
On the other hand, the free commutative differential algebra $\freecd{C_n}$ is isomorphic to
	 \[\bk[\Delta(X)]\big/\left(d_X^m((x^{(0)})^n)-\delta_{m,0}\,\big|\,m\geq0\right),\]
	while $\hat{S}=\left\{(x^{(0)})^n-1\right\}$ is a Gr\"{o}bner-Shirshov basis in $\bk[\Delta(X)]$ when $\lambda\neq0$, such that
	\[\DIrr(\hat{S})=
	 \left\{\prod\nolimits_{i\geq0}(x^{(i)})^{n_i}\,\bigg|\,\sum\nolimits_{i\geq0}n_i<\infty\mbox{ with all }n_i=0,1,\dots,n-1\right\}\]
	is a $\bk$-basis of $\freecd{C_n}$.
\end{coro}

\begin{remark}
	For $\freed{C_n}$ in Corollary \ref{cy}, if weight $\lambda=0$, one can easily see that
	 \[d_X^m((x^{(0)})^n-1)=\sum_{i_1,\dots,i_n\geq0\atop i_1+\cdots+i_n=m}{m\choose i_1,\dots,i_n}x^{(i_1)}\cdots x^{(i_n)}-\delta_{m,0},\]
	for all $m\geq0$, where ${n\choose i_1,\dots,i_n}=\tfrac{n!}{i_1!\cdots i_n!}$ is the multinomial coefficient. In particular,
	\[x^{(m)}(x^{(0)})^{n-1}\equiv -\sum_{i_1=0}^{m-1}\sum_{i_2,\dots,i_n\geq0\atop i_2+\cdots+i_n=m-i_1}{m\choose i_1,\dots,i_n}x^{(i_1)}\cdots x^{(i_n)}-\delta_{m,0}\mod\ideal(\hat{S})\]
	can be used to reduce any words of $\cald_0(C_n)$ into a linear combination of irreducible ones in $\DIrr(\hat{S})$.
	Otherwise, if $\lambda\neq 0$, the expansion of $d_X^m((x^{(0)})^n-1)$ possibly has to be obtained by the more complicated formula \eqref{dm} instead.

	\medskip
	On the other hand, considering $\hat{S}=\left\{(x^{(0)})^n-1\right\}$ in $\bk[\Delta(X)]$ when $\lambda=0$, we in particular have
	 \[[(x^{(0)})^n-1,(x^{(0)})^n-1]_{w_{1,0}}=n^{-1}d_X((x^{(0)})^n-1)x^{(0)}-((x^{(0)})^n-1)x^{(1)}=x^{(1)}\]
	with $w_{1,0}=(x^{(0)})^nx^{(1)}$, but it is nontrivial modulo $(\hat{S},w_{1,0})$, since clearly $x^{(1)}$ can not be any linear combination of differential words $((x^{(0)})^n-1)x^{(0)},\,d_X^m((x^{(0)})^n-1),\,m\geq0$,  in $\DI(\hat{S})$, whose leading terms $\prec w_{1,0}$. Thus
	$\hat{S}$ is not a Gr\"{o}bner-Shirshov basis in $\bk[\Delta(X)]$ when $\lambda=0$.
	
Instead, one can further enlarge $\hat{S}$ to a Gr\"{o}bner-Shirshov basis   \[\hat{S}^+=\left\{(x^{(0)})^n-1,\,x^{(1)}\right\}\]
		in $\bk[\Delta(X)]$ such that
		\[\DIrr(\hat{S}^+)=
		 \left\{(x^{(0)})^k\,\bigg|\,k=0,1,\dots,n-1\right\}\]
		is a $\bk$-basis of $\freecd{C_n}$, as confirmed by Proposition \ref{fcg} saying that $(\freecd{C_n},d_X)$ is isomorphic to $(\bk C_n,0)$. \end{remark}


\noindent
{\bf Acknowledgments.}
This work is supported by the NSFC Grants (Nos. 12071094, 11771142, 11771190) and the China Scholarship Council (No. 201808440068). The authors thank Alexey Ovchinnikov, Gleb Pogudin and William Sit for valuable discussions during the Kolchin Seminar at GC/CUNY. Y. Li also thanks Rutgers University at Newark for providing a stimulating environment of research during his visit from August 2018 to August 2019.

\end{document}